\documentclass[pdflatex,sn-mathphys-num]{sn-jnl}

\usepackage{graphicx}
\usepackage{multirow}
\usepackage{amsmath,amssymb,amsfonts}
\usepackage{amsthm}
\usepackage{mathrsfs}
\usepackage[title]{appendix}
\usepackage{textcomp}
\usepackage{manyfoot}
\usepackage{booktabs}
\usepackage{algorithm}
\usepackage{algorithmicx}
\usepackage{listings}
\usepackage{amscd, fontenc}
\usepackage{calc}
\usepackage{accents}
\usepackage{latexsym,upref,hyperref,wrapfig,graphics,verbatim,enumitem}
\graphicspath{ {./images/} }
\usepackage[utf8]{inputenc}
\usepackage[usenames,dvipsnames]{xcolor}
\usepackage{tikz,lipsum,lmodern}
\usepackage[most]{tcolorbox}
\usepackage{cleveref}
\crefname{lem}{lemma}{Lemma}

\theoremstyle{plain}
\newtheorem{thm}{Theorem}[section]
\newtheorem{theorem}{Theorem}[section]
\theoremstyle{plain}
\newtheorem{lemma}[thm]{Lemma}
\newtheorem{proposition}[thm]{Proposition}
\newtheorem{prop}[thm]{Proposition}

\theoremstyle{definition}

\newcommand{\<}{\left\langle}
\renewcommand{\>}{\right\rangle}
\renewcommand{\(}{\left(}
\renewcommand{\)}{\right)}

\newcommand{\abs}[1]{\left\vert#1\right\vert}

\newcommand{\norm}[1]{\left\Vert#1\right\Vert}

\newcommand{\be} {\begin{equation}}
	\newcommand{\ee} {\end{equation}}
\newcommand{\bea} {\begin{eqnarray}}
	\newcommand{\eea} {\end{eqnarray}}
\newcommand{\Bea} {\begin{eqnarray*}}
	\newcommand{\Eea} {\end{eqnarray*}}

\definecolor{mycolor}{rgb}{0.122, 0.435, 0.698}%







\newtcbtheorem[auto counter,number within=section crefname={theorem}{Theorem},Crefname={Theorem}{Theorem}]%
{Theorem}{Th\smash{e}orem}
{fonttitle=\bfseries\upshape, fontupper=\slshape,
	arc=0mm, colback=blue!5!white,colframe=blue!75!black}{theo}

\newtcbtheorem[auto counter, number within=section, crefname={lemma}{Lemma},Crefname={Lemma}{Lemma}]
{lem}{L\smash{e}mma}%
{fonttitle=\bfseries\upshape, fontupper=\slshape,
	breakable,	arc=0mm, colback=gray!5!white,colframe=green!75!black}{lem}

\newtcbtheorem[auto counter, number within = section, crefname={example}{Example},Crefname={Example}{Example} ]
{Example}{Exam\smash{p}le}{%
	breakable,
	fonttitle = \bfseries,
	colframe = blue!75!black,
	colback = blue!10
}{exa}

\newcommand{\De} {\Delta}
\newcommand{\la} {\lambda}

\newcommand{\grad}{\ensuremath{\nabla}}

\newcommand{\R}{\ensuremath{\mathbb{R}}}

\newcommand{\rN}{\ensuremath{\mathbb{R}^N}}
\newcommand{\Dm}{\ensuremath{\Delta_{g}}}
\newcommand{\M}{\ensuremath{\mathcal{M}}}

\newcommand{\E}{\ensuremath{\mathcal{E}}}

\newcommand{\J}{\mathcal{J}}

\newcommand{\eps}{\varepsilon}

\newcommand{\al} {\alpha}

\newcommand{\de} {\delta}
\newcommand{\ga} {\gamma}

\newcommand{\I}{\infty}

\newcommand{\f}{\frac}

\newcommand{\ml}{\mathcal}

\newcommand{\ef}{\eqref}
\newcommand{\nee}{\notag\ee}
\catcode`\@=11

\makeatletter \@addtoreset{equation}{section} \makeatother
\renewcommand{\theequation}{\arabic{section}.\arabic{equation}}

\begin{document}

\title[Spike layered solutions for elliptic systems on Riemannian Manifolds]{Spike layered solutions for elliptic systems on Riemannian Manifolds}

\author*[1]{Anusree R Kannoth}\email{anusre1412@gmail.com}

\author[2]{Bhakti Bhusan Manna}\email{bbmanna@math.iith.ac.in}
\equalcont{These authors contributed equally to this work.}


\affil{\orgdiv{Department of Mathematics}, \orgname{IIT Hyderabad}, 
	\orgaddress{\street{Kandi}, \city{Sangareddy}, \postcode{502284}, \state{Telangana}, \country{India}}}

\abstract{In this article, we study the following Hamiltonian system:
	\begin{equation*}
		\begin{cases}
			\begin{aligned}
				-\eps^{2}\Dm u +u &= |v|^{q-1}v,\\
				-\eps^{2}\Dm v +v &= |u|^{p-1}u, && \text{ in } \M,	\\
				 \quad u,v &>0, && \text{ in } \M,
			\end{aligned}
		\end{cases}
	\end{equation*}
	where $\M$ is a smooth, compact, and connected Riemannian manifold of dimension $N\ge 3$ without boundary. The exponents $p,q>1$ are assumed to lie below the critical hyperbola, ensuring subcritical growth conditions. We investigate a sequence of least energy critical points of the associated dual functional and analyze their concentration behavior as $\eps\to 0$. Our main result shows that the sequence of solutions exhibits point concentration, with the concentration occurring at a point where the scalar curvature of $\M$ attains its maximum. }



\maketitle

\section{Introduction}

Let $(\M, g)$ be a connected compact smooth Riemannian manifold of dimension $N\geq 3$ without boundary. Consider the singularly perturbed elliptic system 
\begin{equation}\label{eq0}\tag{$P_{\varepsilon}$}
	\begin{cases}
		\begin{aligned}
			-\eps^{2}\Dm u +u &= |v|^{q-1}v,\\ 			
			-\eps^{2}\Dm v +v &= |u|^{p-1}u,  &&\text{ in } \M, \\
			u,v &>0,  &&\text{ in } \M, 
		\end{aligned}
	\end{cases}
\end{equation}
where $\Dm$ is the Laplace-Beltrami operator on $\M$. The exponents $p,q$ satisfy $p,q>1$ and
\begin{equation}\label{HC}\tag{HC}
	\dfrac{1}{p+1}+\dfrac{1}{q+1}>\dfrac{N-2}{N}.
\end{equation}

The corresponding scalar case was studied extensively by several authors. In \cite{MR2180862}, the authors studied the following problem: 
\begin{equation}\label{eq0.2}
	\begin{aligned}
		-&\eps^{2}\Dm u +u = |u|^{p-1}u && \text{ in } \M,\\
	\end{aligned}
\end{equation}
and proved the existence of a mountain pass solution $u_{\eps}$ that exhibits a spike layer, which, as $\eps \to 0$, converges to the maximum point of the scalar curvature of $\M$. The geometry of the domain plays an important role in the concentration phenomena of such perturbed equations. 
Benci et al. proved in \cite{MR2360924} that the existence of solutions to $\eqref{eq0.2}$ depends on the topological property of the manifold by showing that $\eqref{eq0.2}$ has at least cat$(\M)+1$ nontrivial solutions for sufficiently small $\eps$, where cat$(\M)$ denotes the Lusternik-Schnirelmann category of $\M$. In another work by Micheletti and Pistoia \cite{MR2448651}, existence of single-peak solutions concentrating at a stable critical point of the scalar curvature of $\M$ was proved, and in \cite{MR2471314}, Dancer, Micheletti and Pistoia showed the existence of solutions with $k$ peaks that concentrate at an isolated minimum of the scalar curvature of $\M$, as $\eps \to 0$. Some results on the existence and multiplicity of sign-changing solutions are also known; we refer readers to \cite{MR2536955, MR2660847, MR3161661}.

In this work we have adopted the dual variational technique for the existence. This is a standard technique used to prove the existence of ground states. In \cite{MR2183832, MR1971309, MR2499901} the method was considered to study strongly coupled Hamiltonian systems in $\rN$ with singular perturbations. The authors studied the dependence of the concentration profile of the solutions with the given data. In bounded domains, the equation $\eqref{eq0}$ was studied by Avila and Yang in \cite{MR1978382}. They proved the existence of a nontrivial positive solution, with a maximum point approaching a common point on the boundary as $\eps$ goes to zero. Later, Ramos and Yang proposed a more direct method in \cite{MR2135746}, generalizing the result for more general nonlinearities. In \cite{MR2057542}, it was shown that the concentration occurs at the point of maximum of the mean curvature of the boundary, whereas the analogous Dirichlet problem in bounded domains was studied in \cite{MR2408342}. We also refer the readers to \cite{MR3110757}, where a different analytic technique is used to investigate the concentration profile.

There is also an extensive literature exploring the presence of potentials in either the linear or nonlinear components of the system, thus increasing the class of singularly perturbed elliptic systems where concentration phenomena can be analyzed. We refer to the works \cite{MR2199384, MR3645727, MR2966117, MR4826650, MR3559922, MR2730547, MR2229874, MR2338665, MR2287528, MR2342612} where the authors investigated such systems in $\rN$ using techniques like dual variational framework, fractional-order Sobolev space and Nehari manifold techniques, establishing results on existence, multiplicity and concentration for small $\eps>0$.

In this work  we have considered the Hamiltonian elliptic system on a compact and connected Riemannian manifold $\M$ without boundary. We have studied the asymptotic behavior of the least energy critical points of the corresponding dual energy. It is shown that the solutions exhibits single spike patterns and the location for the concentration is at the maximum of the scalar curvature of $\M$.

The limiting profile depends on the solutions of
\begin{equation}\label{eqES}
	\begin{cases}
		\begin{aligned}
			-\Delta u +u &= |v|^{q-1}v,\\
			-\Delta v +v &= |u|^{p-1}u && \text{ in } \rN,
		\end{aligned}
	\end{cases}
\end{equation}
which we generally refer to as the entire system. This equation plays a crucial role in the energy estimate. The ground state solution of $\eqref{eqES}$ exhibits desirable properties, including radial symmetry and exponential decay at infinity. It also determines the first term in the energy expansion. The limiting profile of the perturbed equations is likewise determined by the ground state solution of $\eqref{eqES}$. The major problem to study the asymptotics comes from the non-uniqueness of the solution of the entire system $\eqref{eqES}$. A careful analysis of the size of the solution space is done to address this issue. 

\subsection*{Geometric Preliminaries:}
For $p\in \M$, let $B_g(p,\de):=\{x\in\M: d_g(x,p)<\de\}$ and $T_p(\M)$ be the tangent space of $\M$ at $p$. We denote the exponential map by $exp_p:T_p(\M) \to \M$. Since $\M$ is compact, we have a positive number $\de>0$ independent of the point $p$ and \textit{Riemann normal coordinates} $\Psi_p=(x_1,\dots, x_N): B_g(p,\de)\to \rN$, which is a diffeomorphism on $B_g(p,\de)$. Let us denote the  injectivity radius of $\M$ by $R^*>0$.

Let $\langle \cdot,\cdot \rangle:=\sum_{i,j=1}^{N}g_{ij}dx_i\otimes dx_j$ be the metric tensor. We set $G(x):=(g_{ij}(x))$, $g(x)=\det G(x)$ and $G^{-1}(x)=\(g^{ij}(x)\)$. Also, $d_{g}(p,q)=|\Psi_p(q)|$ provided $q \in B_{g}(p,R^*).$ Then for every $x \in B(0,R^*)=\Psi_{p}(B_{g}(p,R^*))$, we have that (see \cite{MR0768584}),
\begin{equation}\label{metricexpansion}
	\begin{cases}
		\begin{aligned}
			g_{ij}(x)&=\delta_{ij}-\dfrac{1}{3}\<\ml{R}(x, e_{i})x, e_{j}\>+O(|x|^{3}),\\
			g^{ij}(x)&=\delta_{ij}+\dfrac{1}{3}\<\ml{R}(x, e_{i})x, e_{j}\>+O(|x|^{3}),\\
			\sqrt{g}(x)&=1-\dfrac{1}{6}Ric(x, x)+O(|x|^{3}),
		\end{aligned}
	\end{cases}
\end{equation}
where $\mathcal{R}$ and $Ric$ denotes the Riemann curvature tensor and Ricci tensor respectively. The scalar curvature $S(p)$ at a point $p \in \M$ is then defined to be the trace of $Ric$. Also, let us recall, in local coordinates the Laplace-Beltrami operator takes the form
\begin{equation}\label{LapBel1}
	\Dm u(x) =\Delta u(x)+(g_{p}^{ij}(x)-\delta_{ij})\partial_{ij} u-g_{p}^{ij}(x)\Gamma_{ij}^{k}(x)\partial_{k} u,
\end{equation}
where $\Gamma_{ij}^{k}$'s are the Christoffel symbols of the Levi-Civita connection with respect to $g$. And by mean value theorem, we have
\begin{equation}\label{eqChristoffel}
	\Gamma_{ij}^{k}(x)=\Gamma_{ij}^{k}(0)+O(|x|)=O(|x|).
\end{equation} 

\subsection*{Strong Solution:}
By a strong solution to \eqref{eq0}, we mean a pair
\begin{equation*}
	(u_{\eps}, v_{\eps}) \in \E:=W^{2, \tfrac{q+1}{q}}(\M)\times W^{2,\tfrac{p+1}{p}}(\M),
\end{equation*}
satisfying \eqref{eq0} a.e. We define the corresponding energy functional by 
\begin{equation}\label{OEF}
	\J_{\eps}(u, v):=\int_{\M}(\eps^{2}(\nabla_{g} u)\cdot (\nabla_{g} v)+uv)dv_g-\int_{\M}(F(u)+G(v))dv_g,
\end{equation}
where $F(u)=\frac{u^{p+1}}{p+1}$, $G(v)=\frac{v^{q+1}}{q+1}$ and $dv_g$ is the volume form on $\M$. The main result we prove is the following:
\begin{theorem}\label{Thm}
	There exists $\eps_{0}>0$ such that for any $\eps \in (0,\eps_{0})$, the elliptic system $\eqref{eq0}$ has a least energy solution $(u_{\eps}, v_{\eps}) \in \E$ with the following properties:
	\begin{enumerate}[label=(\roman*)]
		\item There exists $K>0$ such that 
		\begin{equation*}
			\max\{\|u_{\eps}\|_{L^{\infty}}, \|v_{\eps}\|_{L^{\infty}}\} \leq K, \qquad \text{ for all } \eps \in (0,\eps_{0}).
		\end{equation*}
		\item For any global maximum points $p_\eps,q_\eps \in\M$ of $u_{\eps}$ and $v_{\eps}$, respectively, we have
		\begin{equation*}
			\frac{d_g(p_\eps,q_\eps)}{\eps} \to 0, \quad \text{ as } \quad \eps \to 0.
		\end{equation*}
		\item The points $p_{\eps}$ and $q_{\eps}$ converge to a unique and common point $p_{0} \in \M$, where the scalar curvature attains its global maximum; i.e.,   
		\begin{equation}
			S(p_{0})=\max_{p \in \M}S(p).
		\end{equation}
	\end{enumerate}
\end{theorem}

This paper is structured as follows: In Section 2, we prove the existence of a critical point using dual variational formulation and also study the solution of the entire problem and its properties. Now, using solution of entire system as test function, we calculate an upper bound for the energy in Section 3. Section 4 is dedicated to studying the behavior of the solution of \eqref{eq0} and a lower energy estimate using this solution. Finally, we end the paper by giving an outline of the bootstrap argument for elliptic system in the Appendix.


\section{The Dual Variational formulation and Existence}

We shall use the dual variational formulation, which is an effective technique in variational methods for handling these kind of coupled systems. A notable advantage of this technique is the Mountain Pass geometry of the level sets. In this section, we explore the fundamental relationship with the natural variational formulation of the system and its dual counterpart. The existence of a critical point for the dual functional then follows as a consequence of the Mountain Pass theorem. Before we derive the dual formulation, we fix some notations. We define
\begin{align} \label{ab}
	\al:=\frac{p+1}{p}, \qquad \beta:=\frac{q+1}{q},
\end{align}
and denote the corresponding spaces
\begin{align*}
	&(I) \qquad X:=L^{p+1}(\M) \times L^{q+1}(\M)   \text{ and, the dual } \,X^{*}=L^{\al}(\M) \times L^{\beta}(\M),\\
	&(II) \quad Y:=L^{p+1}(\rN) \times L^{q+1}(\rN)    \text{ and, the dual } \,Y^{*}=L^{\al}(\rN) \times L^{\beta}(\rN),
\end{align*}
endowed with the norms:
\begin{align*}
	&\|w\|^{2}_{A\times B}=\|w_{1}\|^{2}_{A}+\|w_{2}\|^{2}_{B}, \quad w=(w_{1},w_{2}) \in A\times B.
\end{align*}
From \eqref{HC}, we have
\begin{equation}\label{eqExponent}
	p+1<\beta^{*}=\dfrac{N(q+1)}{Nq-2(q+1)} \text{ and } 	q+1<\alpha^{*}=\dfrac{N(p+1)}{Np-2(p+1)}.
\end{equation}
Since $\M$ is a compact Riemannian manifold, we have that the inclusions
\begin{equation*}
	i_{p}: W^{2,\frac{p+1}{p}}(\M) \hookrightarrow L^{q+1}(\M) \quad \text{and } \quad i_{q}:W^{2,\frac{q+1}{q}}(\M)\hookrightarrow L^{p+1}(\M),
\end{equation*}
are compact (Theorem 3.6, \cite{MR1481970}). Now to set the dual variational formulation we use the following lemma: 
\begin{lemma}\label{LP-estimate}
	For any $p\in(1,\infty)$ and $f\in L^p(\M)$, the problem
	\be -\De_g u +u=f \text{ in } \M,\nee 
	has a unique solution $u\in W^{2,p}(\M)$. Moreover, for some positive constant $C$ independent of $f$,
	\be \norm{u}_{W^{2,p}(\M)}\le C\norm{f}_{L^p(\M)}.\nee
\end{lemma}

\begin{proof}
	The result can be easily verified using local co-ordinate systems, partition of unity and the standard $L^p$-estimate.
\end{proof}

Define $T_{1,p}: L^{\al}(\M) \to W^{2,\al}(\M)$ and $T_{1,q}: L^{\beta}(\M) \to W^{2,\beta}(\M)$  as
\begin{equation*}
	T_{1,p}=T_{1,q} := \left( -\varepsilon^{2}\Dm+\mathit{id}\right) ^{-1}.
\end{equation*}
Then the compositions $T_{1}=i_{q}\circ T_{1,q}$, $T_{2}=i_{p}\circ T_{1,p}$ given by
\begin{equation*}
	T_{1}: L^{\beta}(\M) \to L^{p+1}(\M) \quad \text{ and }\quad T_{2}: L^{\al}(\M) \to L^{q+1}(\M),
\end{equation*}
are compact, linear and self-adjoint. Let us define
\begin{equation*}
	T_{\varepsilon} : X^{*} \to X \quad \text{ as }\quad T_{\varepsilon}:=
	\begin{pmatrix}
		0 & T_{1}\\
		T_{2} & 0
	\end{pmatrix}.
\end{equation*}
Precisely, $T_{\varepsilon}$ is defined by
\begin{equation}
	\left\langle T_{\varepsilon}w,\eta\right\rangle = \eta_{1}T_{1}w_{2}+ \eta_{2}T_{2}w_{1},
\end{equation}  
where $\eta=(\eta_{1},\eta_{2})\in X^{*}, \; w=(w_{1},w_{2}) \in X^{*}$. We define the dual functional on $X^{*}$ as
\begin{equation*}
	\begin{aligned}
		I_{\eps}(w) :=&\int_{\M}\dfrac{p}{p+1}|w_{1}|^{\tfrac{p+1}{p}}+\dfrac{q}{q+1}|w_{2}|^{\tfrac{q+1}{q}}dv_g -\dfrac{1}{2}\int_{\M}\left( w_{1}\mathit{T_{1}}w_{2}+w_{2}\mathit{T_{2}}w_{1}\right) dv_g,
	\end{aligned}
\end{equation*}
where $w=(w_{1},w_{2}) \in X^{*}.$ $I_{\eps}(w)$ is $C^{1}$ with derivative
\be 
I^{\prime}_\eps(w_1,w_2)(h,k)=\int_{\M}|w_1|^{\f{1-p}{p}}w_1h+\int_{\M}|w_2|^{\f{1-q}{q}}w_2k-\int_{\M}<T_\eps(w_1,w_2),(h,k)>, \quad \forall \; (h, k) \in X^{*}.
\nee We have the following comparison lemma on $I_{\eps}$ and $J_{\eps}$.

\begin{lemma}\label{comp.dual.soln}
	Let $(w_1,w_2)\in X^*$ be a critical point of $I_\eps$ and $(u_\eps,v_\eps):=(T_1 w_2,T_2 w_1)$ be the corresponding dual function. Then 
	\begin{enumerate}[label= (\roman*)]
		\item $(u_\eps,v_\eps)$ solves \ef{eq0} and $w_1=|u_\eps|^{p-1}u_\eps$, $w_2=|v_\eps|^{q-1}v_\eps$,
		\item $I_\eps(w_1,w_2)=J_\eps(u_\eps,v_\eps)$.
		\item For any solution $(u_\eps,v_\eps)$ of \ef{eq0} the corresponding dual $(w_1,w_2)$ defined in (i) is a critical point of the dual functional $I_\eps$. 
	\end{enumerate} 
\end{lemma}

\begin{proof}
	\textbf{Proof of (i):} 
	First, we note that if $ w=(w_1,w_2) \in X^*$ is a critical point of $I_{\eps}$, then
	\be 
	\int_{\M}|w_1|^{\f{1-p}{p}}w_1h+\int_{\M}|w_2|^{\f{1-q}{q}}w_2k=\int_{\M}h T_1 w_2+ k T_2 w_1 . \quad \forall \; (h,k)\in X^*.
	\nee Hence,
	\be \label{dual.trans}
	|w_1|^{\f{1-p}{p}}w_1=  T_1 w_2=u_{\eps},\qquad |w_2|^{\f{1-q}{q}}w_2= T_2 w_1 =v_{\eps}.
	\ee
	This gives $w_{1}=|u_\eps|^{p-1}u_\eps$ and $w_{2}=|v_\eps|^{q-1}v_\eps$. Also,
	
	\begin{equation*}
		(u_\eps,v_\eps)=(T_1 w_2,T_2 w_1) \implies \begin{cases}
			w_{1}=-\eps^2\De v_\eps + v_\eps,\\ w_{2}=-\eps^2\De u_\eps +u_\eps.
		\end{cases}
	\end{equation*}
	Hence $(u_\eps,v_\eps)$ satisfies
	\be \label{mot.sol}
	-\eps^2\De u +u=|v|^{q-1}v,  \qquad -\eps^2\De v + v= |u|^{p-1}u. 
	\ee
	So, for every positive critical points of $J_\eps$, the corresponding dual pair, given in \ef{dual.trans} is positive and solves the equation \eqref{eq0}.  This proves $(i).$ \\
	\textbf{Proof of (ii):} 
	We note that for any critical point $(w_1,w_2)$ of $I_\eps$,
	\begin{align*}
		I_\eps(w_1,w_2)&=\f{p}{p+1}\int_{\M}|w_1|^{\f{p+1}{p}}+\f{q}{q+1}\int_{\M}|w_2|^{\f{q+1}{q}}-\f{1}{2}\int_{\M}<T_\eps(w_1,w_2),(w_1,w_2)>\\
		&=\left( \f{p}{p+1}-\f{1}{2}\right) \int_{\M}|u_{\eps}|^{p+1}dv_g+\left( \f{q}{q+1}-\f{1}{2}\right) \int_{\M}|v_{\eps}|^{q+1}dv_g  \\
		&=\left( \f{1}{2}-\f{1}{p+1}\right) \int_{\M}|u_{\eps}|^{p+1}dv_g+\left( \f{1}{2}-\f{1}{q+1}\right) \int_{\M}|v_{\eps}|^{q+1}dv_g\\
		&=J_\eps(u_{\eps}, v_{\eps}). 	
	\end{align*}
	
	\textbf{Proof of (iii):}  This proof is quite straight forward. For any solution $ (u_\eps,v_\eps)\in\E$ of \eqref{eq0}, define $w_1=|u_\eps|^{p-1}u_\eps$ and $w_2=|v_\eps|^{q-1}v_\eps$. Then $(w_1,w_2)\in X^*$ and we can show for any $\phi,\psi\in C_c^\infty(\M)$, $\; I^{\prime}_\eps(w_1,w_2)(\phi,\psi)=0$.
\end{proof}  

\subsection*{Mountain Pass Theorem and the Existence:}
Arguments similar to those in Lemma 1 of \cite{MR1944783} and Lemma 2.2 of \cite{MR1978382}, establish the following Mountain Pass geometry of the dual functional $I_{\eps}$.

\begin{lemma}\label{MPG}
	The functional $I_{\eps}$ has the Mountain Pass Geometry, namely
	\begin{enumerate}[label=(\roman*)]
		\item There exist $\rho, a>0$ such that $I_{\eps}(w)\geq a$, if  $\|w\|_{X^{*}}=\rho$,
		\item There exists $e \in X^{*}$ with $\|e\|_{X^{*}}>\rho$ such that $I_{\eps}(e)\leq 0$.
	\end{enumerate}
\end{lemma}

Also, we have the following:
\begin{lemma}\label{PS}
	The functional $I_{\eps}$ satisfies the Palais Smale condition.
\end{lemma}

\begin{proof}
	Let $w_{n}=(w_{n, 1}, w_{n, 2}) \in X^{*}$ be a Palais Smale sequence, i.e.,
	\begin{equation*}
		|I_{\eps}(w_{n, 1}, w_{n, 2})| \leq C \quad \text{ and} \quad I^{\prime}_{\eps}(w_{n, 1}, w_{n, 2}) \to 0.
	\end{equation*}
	\textbf{Claim 1:} $(w_{n, 1}, w_{n, 2})$ is bounded in $X^{*}$.
	
	Since $|I_{\eps}(w_{n, 1}, w_{n, 2})| \leq C$, we can write
	\begin{equation*}
		\begin{aligned}
			\int_{\M}\dfrac{p}{p+1}|w_{n, 1}|^{\tfrac{p+1}{p}}+\dfrac{q}{q+1}|w_{n, 2}|^{\tfrac{q+1}{q}}dv_g &\leq\dfrac{1}{2}\int_{\M}\left( w_{n, 1}\mathit{T_{1}}w_{n, 2}+w_{n, 2}\mathit{T_{2}}w_{n, 1}\right) dv_g+C\\
			&\leq\dfrac{1}{2}\int_{\M}\left( |w_{n, 1}|^{\tfrac{p+1}{p}}+|w_{n, 2}|^{\tfrac{q+1}{q}}\right) dv_g+C+o(1).
		\end{aligned}
	\end{equation*}
	This implies
	\begin{equation*}
		\begin{aligned}
			\(\dfrac{p}{p+1}-\dfrac{1}{2}\)\int_{\M}|w_{n, 1}|^{\tfrac{p+1}{p}}dv_g+\(\dfrac{q}{q+1}-\dfrac{1}{2}\)\int_{\M}|w_{n, 2}|^{\tfrac{q+1}{q}}dv_g
			&\leq C +o(1).
		\end{aligned}
	\end{equation*}
	Thus $(w_{n, 1}, w_{n, 2})$ is bounded in $X^{*}$. Since  $X^{*}$ is reflexive, we may assume up to a subsequence that $(w_{n, 1}, w_{n, 2}) \rightharpoonup (w_{1}, w_{2})$ weakly in $X^{*}$. \\
	\textbf{Claim 2:} $(w_{n, 1}, w_{n, 2})$ converges strongly in $X^{*}$.
	
	Let $(u_{n}, v_{n}):=(T_{1}w_{n, 2}, T_{2}w_{n, 1})$, and $(u, v):=(T_{1}w_{2}, T_{2}w_{1})$. 
	Since $T_{1}: L^{\frac{q+1}{q}}(\M) \to L^{p+1}(\M)$ and $T_{2}: L^{\frac{p+1}{p}}(\M) \to L^{q+1}(\M)$ are compact operators, we have (up to a subsequence)
	\begin{equation*}
		(u_{n}, v_{n})=(T_{1}w_{n, 2}, T_{2}w_{n, 1})\to (T_{1}w_{2}, T_{2}w_{1})=(u,v) \text{ strongly in } L^{p+1}(\M) \times L^{q+1}(\M).
	\end{equation*}
	We now recall the following standard inequality: for all $a, b \in \rN$ and $p>1$,
	\begin{equation}
		\left||a|^{p-1}a-|b|^{p-1}b\right| \leq
		C(p)|a-b|(|a|^{p-1}+|b|^{p-1}),
	\end{equation}
	where $C(p)$ is a constant depending on $p$ (proof can be found in \cite{AB}). Using this,
	\begin{equation*}
		\begin{aligned}
			\int_{\M}\left||u_{n}|^{p-1}u_{n}-|u|^{p-1}u\right|^{\frac{p+1}{p}}dv_g
			&\leq c\int_{\M}|u_{n}-u|^{\frac{p+1}{p}}||u_{n}|^{p-1}+|u|^{p-1}|^\frac{p+1}{p} dv_g\\
			&\leq c(\int_{\M}|u_{n}-u|^{p+1})^{\tfrac{1}{p}}(\int_{\M}\left||u_{n}|^{p-1}+|u|^{p-1}\right|^{\tfrac{p+1}{p-1}})^{\tfrac{p-1}{p}}\\
			&\leq c(\int_{\M}|u_{n}-u|^{p+1})^{\tfrac{1}{p}}(\int_{\M}|u_{n}|^{p+1}+\int_{\M}|u|^{p+1})^{\tfrac{p-1}{p}}\\
			& \to 0 \quad \text{ strongly}.	
		\end{aligned}
	\end{equation*}
	Thus $|u_{n}|^{p-1}u_{n} \to |u|^{p-1}u$ and $|v_{n}|^{q-1}v_{n} \to |v|^{q-1}v$ in $L^{\frac{p+1}{p}}(\M)$ and $L^{\frac{q+1}{q}}(\M)$ respectively. Also, $I^{\prime}_{\eps}(w_{n, 1}, w_{n, 2})\to 0$ implies
	\begin{equation*}
		u_{n}=|w_{n, 1}|^{\tfrac{1}{p}-1}w_{n, 1}+o(1) \;\text{ in } L^{p+1}(\M), \quad v_{n}=|w_{n, 2}|^{\tfrac{1}{q}-1}w_{n, 2}+o(1)  \;\text{ in } L^{q+1}(\M).
	\end{equation*}
	Now using the relation $(u_{n}, v_{n}):=(T_{1}w_{n, 2}, T_{2}w_{n, 1})$, we conclude that
	\begin{equation*}
		\begin{cases}
			\begin{aligned}
				-\eps^{2}\Delta u_{n}+u_{n}&=w_{n, 2}=|v_{n}|^{q-1}v_{n}+o(1) \;\text{ in } L^{\frac{p+1}{p}}(\M),\\
				-\eps^{2}\Delta v_{n}+v_{n}&=w_{n, 1}=|u_{n}|^{p-1}u_{n}+o(1) \;\text{ in } L^{\frac{q+1}{q}}(\M). 
			\end{aligned}
		\end{cases}
	\end{equation*}
	This immediately gives $w_{1}=|u|^{p-1}u$, $w_{2}=|v|^{q-1}v$ and
	\begin{equation*}
		\|w_{n}-w\|^{2}_{X^{*}}=\||u_{n}|^{p-1}u_{n}-|u|^{p-1}u\|^{2}_{L^{\frac{p+1}{p}}(\M)}+\||v_{n}|^{q-1}v_{n}-|v|^{q-1}v\|^{2}_{L^{\frac{q+1}{q}}(\M)} +o(1) \to 0 \quad \text{ strongly}.
	\end{equation*}
\end{proof}

Define
\begin{equation}
	c_{\eps} :=\inf_{\gamma \in \Gamma}  \max_{0 \leq t \leq 1}I_{\eps}(\gamma(t))\label{lstenergy},
\end{equation}
where
\begin{equation*}
	\Gamma := \{\gamma \in C([0,1];X^{*})|I_{\eps}(\gamma(0))=0, I_{\eps}(\gamma(1))<0\}.
\end{equation*}

\begin{prop}
	The functional $I_{\eps}$ has a critical point $w_{\eps}$ such that $I_{\eps}(w_{\eps})=c_{\eps}$.
\end{prop}
\begin{proof}
	Since the functional satisfies \cref{MPG} and \cref{PS}, by Mountain Pass Lemma we obtain a critical point $w_{\eps}$ such that $I_{\eps}(w_{\eps})=c_{\eps}$.
\end{proof}

We now study the entire system corresponding to our equation.

\subsection*{Corresponding Entire Problem:}

Consider the system
\begin{equation}\label{es}
	\begin{cases}
		\begin{aligned}
			-\Delta u +u &=|v|^{q-1}v,\\
			-\Delta v +v &=|u|^{p-1}u &&\quad \text{ in } \mathbb{R}^{N},\\
			u,v&>0 &&\quad \text{ in } \mathbb{R}^{N}.
		\end{aligned}
	\end{cases}
\end{equation}

In \cite{MR1755067, MR1785681}, the authors have proved the existence of least energy solutions for this system, as well as the symmetry and decay properties of weak solutions. In this works, the least energy is defined with respect to an energy functional that incorporates canonical isomorphism between the fractional-order Sobolev space and $L^2$. In contrast, our work reformulates the problem using dual variational approach and this structural difference leads to distinct analytical frameworks. Here, we address separately the existence of least energy solutions in our setting, along with an analysis of their symmetry and decay properties. First we recall the following:
\begin{lemma}[Lemma A, \cite{MR2183832}]\label{Lemma_A}
	Let $s\in(1,\I)$. Then for every $h\in L^s(\R^N)$ the problem
	\be	- \De u +  u= h\quad \text{  in  } \R^N, \nee
	has a unique solution $u\in W^{2,s}(\R^N)$. Furthermore, there exists a constant $K>0$ such that 
	\be \norm{u}_{W^{2,s}(\R^N)}\leq K(s)\norm{h}_{L^s(\R^N)}.  \ee 
\end{lemma}

Let us consider the following linear continuous operator \cite{MR125307}:
\be S_{p}:=(-\De+id)^{-1} :L^{p}(\rN)\to W^{2,p}(\rN) \quad \text{ for } \quad 1<p<\infty,\label{Ent.est1}\ee
and the continuous embeddings
\begin{equation*}
	\mathit{j}_{p}: W^{2,\al}(\rN) \to L^{q+1}(\rN) \quad \text{ and } \quad \mathit{j}_{q}: W^{2,\beta}(\rN) \to L^{p+1}(\rN).
\end{equation*}
Then 
\begin{equation*}
	S_{1}:=j_q\circ S_{\beta}: L^{\beta}(\rN) \to L^{p+1}(\rN), \quad \text{ and } \quad S_{2}:= j_p\circ S_{\alpha}:L^{\al}(\rN) \to L^{q+1}(\rN),
\end{equation*}
are also linear continuous operators. As in the previous section we can define the dual energy $I_\infty$ on $Y^*:=L^{\frac{p+1}{p}}(\rN)\times L^{\frac{q+1}{q}}(\rN)$ as 
\begin{equation}
	I_\infty(w_1,w_2):=\int_{\rN}\dfrac{p}{p+1}|w_{1}|^{\tfrac{p+1}{p}}+\dfrac{q}{q+1}|w_{2}|^{\tfrac{q+1}{q}}dx -\dfrac{1}{2}\int_{\rN}\left( w_{1}\mathit{S_{1}}w_{2}+w_{2}\mathit{S_{2}}w_{1}\right) dx.
\end{equation}

Note that the same results in \cref{comp.dual.soln} hold true in this case as well. Hence, for any positive radial solution of the limit problem, the corresponding dual pair is also positive radial and a critical point of the dual energy, and vice-versa. Also both the energy functionals attain the same value at these corresponding solution.

Now we see from \ef{Ent.est1} that the energy $I_\infty$ has Mountain-Pass geometry. The proof follows exactly as in \cref{MPG}. Let $\ml{C}_\infty$ be the min-max critical level of $I_\infty$ defined as 
\be \ml{C}_\infty:=\inf_{\varphi\in\ml{P}} \sup_{t\in[0,1]} I_\infty(\varphi(t)), \nee where
\be \ml{P}:=\{ \varphi \in C([0,1],  Y^*) \, |\, I_\infty(\varphi(0))=0,  I_\infty(\varphi(1))<0   \}.\nee

\begin{lemma}\label{PS-soln}
	There exist a positive critical point $(h,k)$ of $I_\infty$ with $I_\infty(h,k)=\ml{C}_\infty$.
\end{lemma}
\begin{proof}
	Let $(h_n,k_n)\in Y^*$ be a $(PS)$ sequence at the level $\ml{C}_{\infty}$. Then,
	\begin{equation}\label{LowerBound}
		\(\dfrac{p}{p+1}-\dfrac{1}{2}\)\int_{\rN}|h_{n}|^{\tfrac{p+1}{p}}dx+\(\dfrac{q}{q+1}-\dfrac{1}{2}\)\int_{\rN}|k_{n}|^{\tfrac{q+1}{q}}dx = \mathcal{C}_{\infty} +o(1).
	\end{equation}
	This implies $(h_n, k_n)$ is bounded in $Y^{*}$, and there exists $(h,k)\in Y^{*}$ such that  up to a subsequence, $(h_n,k_n)\rightharpoonup (h,k)$ in $Y^*$. Also, the above equation directly implies that $(h_n, k_n)\nrightarrow (0,0)$. Then by Lion's Lemma, we get $r,\eta>0$ and a sequence $\{x_n\}_{n}$ in $\rN$ such that 
	\begin{equation}\label{non-van-ps-seq} 
		\liminf_{n \to +\infty}\int_{B(x_{n},R)}|h_n|^{\frac{p+1}{p}}\geq \eta\quad  \text{ or }\quad \liminf_{n \to +\infty}\int_{B(x_{n},R)}|k_n|^{\frac{q+1}{q}}\geq \eta .
	\end{equation}
	Let $(u_{n}, v_{n}) :=(S_{1}k_{n}, S_{2}h_{n})$. Then, by the continuity of the operators we have 
	\begin{equation*}
		\|u_{n}\|_{W^{2, \beta}(\rN)}, \|v_{n}\|_{W^{2, \alpha}(\rN)} \leq C,
	\end{equation*}
	for some positive constant $C$. Therefore, up to a subsequence, $(u_n,v_n)\rightharpoonup (u,v)$ in $W^{2, \beta}(\rN)\times W^{2, \alpha}(\rN)$, and by Rellich theorem
	\begin{equation*}
		(u_{n}, v_{n}) \to (u, v) \quad \text{ strongly in } \quad L^{p+1}_{loc}(\rN) \times L^{q+1}_{loc}(\rN),
	\end{equation*}
	Also, for any $(\phi,\psi)\in Y^*$, we have 
	\begin{align*}
		I^{{\prime}}_\infty(h_n,k_n)(\phi,\psi)=o(1)
		\implies&\int_{\rN}\(|h_n|^{\al-2}h_n- u_n\)\phi dx +\int_{\rN}\(|k_n|^{\beta-2}k_n- v_n\)\psi dx =o(1).
	\end{align*}
	Replacing $(\phi,\psi)$ by $(\phi,0)$ and $(0,\psi)$ we get 
	\be \int_{\rN}\(|h_n|^{\al-2}h_n- u_n\)\phi dx=o(1)=\int_{\rN}\(|k_n|^{\beta-2}k_n- v_n\)\psi dx.\label{lpq.eq}\ee
	This implies 
	\begin{equation*}
		|h_{n}|^{\tfrac{1}{p}-1}h_{n} \rightharpoonup u, \quad |k_n|^{\tfrac{1}{q}-1}k_n \rightharpoonup v \quad \text{ in } L^{p+1}(\rN) \text{ and } L^{q+1}(\rN) \text{ respectively}.
	\end{equation*}
	Define the functions
	\begin{equation*}
		(\tilde{h}, \tilde{k}):=(|u|^{p-1}u, |v|^{q-1}v) \qquad \text{ so that } \quad (u, v)=(|\tilde{h}|^{\tfrac{1}{p}-1}\tilde{h}, |\tilde{k}|^{\tfrac{1}{q}-1}\tilde{k}).
	\end{equation*}
    Now, let
    \be  \phi:=\abs{|h_n|^{\al-2}h_n- u_n}^{p-1}\(|h_n|^{\al-2}h_n- u_n\), \text{ and }\psi:=\abs{|k_n|^{\beta-2}k_n- v_n}^{q-1}\(|k_n|^{\beta-2}k_n- v_n\).\nee  
    Then $(\phi,\psi) \in Y^{*}$ (see Proposition 4.1 of \cite{MR4478364}). Using this test function in \eqref{lpq.eq}, we get
	\be \norm{|h_n|^{\al-2}h_n- u_n}_{L^{p+1}(\rN)}=o(1)=\norm{|k_n|^{\beta-2}k_n- v_n}_{L^{q+1}(\rN)}.\nee
	Then,
	\begin{equation*}
		\norm{|h_n|^{\tfrac{1}{p}-1}h_n- |\tilde{h}|^{\tfrac{1}{p}-1}\tilde{h}}_{L^{p+1}_{loc}(\rN)} \leq \norm{|h_n|^{\tfrac{1}{p}-1}h_n- u_{n}}_{L^{p+1}_{loc}(\rN)}+\norm{u_{n}- u}_{L^{p+1}_{loc}(\rN)}=o(1).
	\end{equation*}
	Therefore, $h_{n} \to \tilde{h}$ a.e. in $\rN$ and similarly, $k_{n} \to \tilde{k}$ a.e. in $\rN.$ From the uniqueness of weak limits, it follows that $h=\tilde{h}$ and $k=\tilde{k}$. And hence  
	\be \left| \int_{\rN}|h_{n}|^{\tfrac{1}{p}-1}h_{n}\phi dx-\int_{\rN}|h|^{\tfrac{1}{p}-1}h\phi dx\right|=o(1)=\left|\int_{\rN}|k_{n}|^{\tfrac{1}{q}-1}k_{n}\psi dx-\int_{\rN}|k|^{\tfrac{1}{q}-1}k\psi dx\right|.\nee
  Finally, using the continuity of the operators $S_1$ and $S_2$ we get 
	\be \left|\int_{\rN}\left( \phi S_{1}k_{n}+\psi S_{2}h_{n}\right) dx-\int_{\rN}\left( \phi S_{1}k+\psi S_{2}h\right) dx\right|=o(1).\nee
	Combining all these arguments we get
	\begin{equation*}
		\begin{aligned}
			&|I_\infty^{\prime}(h_{n},k_{n})(\phi, \psi)-I_\infty^{\prime}(h,k)(\phi,\psi)|\\
			&=\left| \int_{\rN}|h_{n}|^{\tfrac{1}{p}-1}h_{n}\phi dx+\int_{\rN}|k_{n}|^{\tfrac{1}{q}-1}k_{n}\psi dx
			-\int_{\rN}\left( \phi S_{1}k_{n}+\psi S_{2}h_{n}\right) dx\right.\\
			&- \left.\int_{\rN}|h|^{\tfrac{1}{p}-1}h\phi dx-\int_{\rN}|k|^{\tfrac{1}{q}-1}k\psi dx+\int_{\rN}\left( \phi S_{1}k+\psi S_{2}h\right) dx\right|\\
			&\leq \left| \int_{\rN}|h_{n}|^{\tfrac{1}{p}-1}h_{n}\phi dx-\int_{\rN}|h|^{\tfrac{1}{p}-1}h\phi dx\right|+\left|\int_{\rN}|k_{n}|^{\tfrac{1}{q}-1}k_{n}\psi dx-\int_{\rN}|k|^{\tfrac{1}{q}-1}k\psi dx\right|\\
			&+\left|\int_{\rN}\left( \phi S_{1}k_{n}+\psi S_{2}h_{n}\right) dx-\int_{\rN}\left( \phi S_{1}k+\psi S_{2}h\right) dx\right|\\
			&=o(1).
		\end{aligned}
	\end{equation*}
	Hence we have found a critical point $(h,k)$ of $I_\infty.$ Now, if $(h,k)=(0,0)$, define 
	\begin{equation*}
		(\tilde{h}_{n}(x), \tilde{k}_{n}(x))=(h_{n}(x+x_{n}), k_{n}(x+x_{n})).
	\end{equation*}
	where $\{x_n\}$ is the sequence defined in \eqref{non-van-ps-seq}. Then we have, 
	\begin{equation*}
		I_\infty(\tilde{h}_{n}, \tilde{k}_{n}) = I_\infty(h_{n}, k_{n}) \to \mathcal{C}_{\infty},  \qquad \text{and } \quad I_\infty^{\prime}(\tilde{h}_{n}, \tilde{k}_{n})(\phi, \psi) = I_\infty^{\prime}(h_{n}, k_{n})(\tilde{\phi}, \tilde{\psi}) = 0,
	\end{equation*}
	where $\tilde{\phi}(x)=\phi(x-x_{n})$ and $\tilde{\psi}(x)=\psi(x-x_{n})$. Then $(\tilde{h}_{n}, \tilde{k}_{n})$ is a bounded sequence and $(\tilde{h}_{n}, \tilde{k}_{n}) \rightharpoonup  (\tilde{h}, \tilde{k})$. Also by \eqref{non-van-ps-seq}, we get
	\begin{equation}
		\begin{aligned}
			\liminf_{n \to +\infty}\int_{B(0,R)}|\tilde{h}_{n}|^{\frac{p+1}{p}}dx &=\liminf_{n \to +\infty}\int_{B(x_{n},R)}|h_{n}|^{\frac{p+1}{p}}dx \geq \eta>0,\\
			\text{or}  \qquad \liminf_{n \to +\infty}\int_{B(0, R)}|\tilde{k}_{n}|^{\frac{q+1}{q}}dx &=\liminf_{n \to +\infty}\int_{B(x_{n},R)}|k_{n}|^{\frac{q+1}{q}}dx \geq \eta>0.
		\end{aligned}
	\end{equation}
	This gives $(\tilde{h}, \tilde{k})$ is a nontrivial solution. Thus we have for the nontrivial solution $(h, k)$, $\mathcal{C}_{\infty} \leq I_{\infty}(h, k)$. Also
	\begin{equation*}
		\begin{aligned}
			I_{\infty}(h,k) &=\(\dfrac{p}{p+1}-\dfrac{1}{2}\)\int_{\rN}|h|^{\tfrac{p+1}{p}} dx+\(\dfrac{q}{q+1}-\dfrac{1}{2}\)\int_{\rN}|k|^{\tfrac{q+1}{q}} dx\\
			&\leq\liminf_{n \to \infty}\left[ \(\dfrac{p}{p+1}-\dfrac{1}{2}\)\int_{\rN}|h_{n}|^{\tfrac{p+1}{p}} dx+\(\dfrac{q}{q+1}-\dfrac{1}{2}\)\int_{\rN}|k_{n}|^{\tfrac{q+1}{q}} dx\right] \\
			&=\liminf_{n \to \infty}\left[ I_{\infty}(h_{n},k_{n}) - \dfrac{1}{2}I_{\infty}^{\prime}(h,k)(h,k)\right] \\
			&=\mathcal{C}_{\infty}.
		\end{aligned}
	\end{equation*}
	It then follows from Lemma 3.1 of \cite{MR2287528} that the ground state $(h, k)$ is positive.
\end{proof}

Let $\tilde{\ml{L}}$ denote the set of all positive critical points of $I_\infty$. For any $(h,k)\in \tilde{\ml{L}}$, define $(U,V):=(S_1k,S_2h)\in W^{2,\beta}(\rN)\times W^{2,\alpha}(\rN)$. Then $(U,V)$ is a positive solution of the entire problem \ef{es}, with $(U,V)=(h^{1/p},k^{1/q})$. From \cite{MR1755067, MR1785681}, every $(h,k)\in \tilde{\ml{L}}$ is radial, and the corresponding $(U,V)$ is also radial and satisfies the decay estimates
\begin{equation}
	|D^{\delta}U(x)|, |D^{\delta}V(x)| \leq C exp(-c|x|),\label{e.decay}
\end{equation}
for all $\abs{\delta}\le 2$, $C, c>0$. We denote the set  
\be \ml{L}:=\{(U,V)\in W^{2,\beta}(\rN)\times W^{2,\alpha}(\rN)\,|\,  (U^p,V^q) \text{ is a least energy positive critical point of } I_\infty  \}.\nee

\begin{proposition}\label{cptsolnsp}
	$\ml{L}$ is compact in $Y^*$.
\end{proposition}
\begin{proof}
	As the embedding $W^{2,\al}_r(\rN)\hookrightarrow L^{s}(\rN)$ is compact for all $\al <s<\al^*$, we have $W_r^{2,\beta}(\rN)\times W_r^{2,\alpha}(\rN)\hookrightarrow L^{p+1}(\rN) \times L^{q+1}(\rN)$ is compact. So it is enough to prove $\ml{L}$ is a closed subset of $W_r^{2,\beta}(\rN)\times W_r^{2,\alpha}(\rN)$. Let $(U_n, V_n)\in \ml{L}$ converge to $(U,V)$ in $W^{2,\beta}(\rN)\times W^{2,\alpha}(\rN).$ Then obviously $(U^p_n, V^q_n)$ is a $(PS)$ sequence for $I_\infty$, and using the same arguments in \cref{PS-soln} we can easily show $(U,V)\in\ml{L}$.
\end{proof}


\section{An upper bound for the Least energy}

In this section, we derive an appropriate upper bound for the least energy. This bound is crucial as it provides important insight into the concentration behaviour and the limiting profile of the least energy solution. Let $0<R<R^*$. Consider a radial cutoff function $\phi_{R} \in C_{0}^{\infty}\(\rN,[0,1]\)$ such that $|\grad \phi_{R}| \leq \frac{2}{R}$, $|\Delta \phi_{R}| \leq \frac{2}{R^{2}}$ and
\begin{equation}\label{radcutoff}
	\phi_{R}(x)=
	\begin{cases}
		1 \quad \text{ for } |x| \leq \frac{R}{2},\\
		0 \quad \text{ for } |x| \geq R.
	\end{cases}
\end{equation}
For $p \in M$ and $\eps>0$, we define the function
\begin{equation*}
	\(U_{\eps}^{R}(x), V_{\eps}^{R}(x)\):=
	\begin{cases}
		\begin{aligned}
			&\(\phi_{R}\(\Psi_{p}(x)\)U_{\eps}\(\Psi_{p}(x)\),  \phi_{R}\(\Psi_p(x)\)V_{\eps}\(\Psi_{p}(x)\)\) &&\text{if } x\in B_{g}(p,R),\\
			&0, &&\text{otherwise},
		\end{aligned}
	\end{cases}
\end{equation*}
where $U_{\eps}(x)=U(\dfrac{x}{\eps})$ and $V_{\eps}(x)=V(\dfrac{x}{\eps})$ for a least energy solution $(U, V)$ of $\eqref{es}$. For convenience, we denote $\(U_{\eps}^{R}(x), V_{\eps}^{R}(x)\)$ by $\(U_{*}(x), V_{*}(x)\)$. Let $w=\(w_{1}, w_{2}\)= \(U_{*}^{p}(x), V_{*}^{q}(x)\)$ and define
\begin{equation*}
	\begin{aligned}
		h_\eps(t):=I_{\eps}(tw) &=\int_{\M}\dfrac{pt^{\tfrac{p+1}{p}}}{p+1}|w_{1}|^{\tfrac{p+1}{p}}+\dfrac{qt^{\tfrac{q+1}{q}}}{q+1}|w_{2}|^{\tfrac{q+1}{q}}dv_g -\dfrac{t^{2}}{2}\int_{\M}\left( w_{1}\mathit{T_{1}}w_{2}+w_{2}\mathit{T_{2}}w_{1}\right) dv_g\\
		&=\dfrac{pt^{\tfrac{p+1}{p}}}{p+1}\int_{\M}U_{*}^{p+1}dv_g+\dfrac{qt^{\tfrac{q+1}{q}}}{q+1}\int_{\M}V_{*}^{q+1}dv_g-\dfrac{t^{2}}{2}\int_{\M}\left( U_{*}^{p}\mathit{T_{1}}V_{*}^{q}+V_{*}^{q}\mathit{T_{2}}U_{*}^{p}\right) dv_g.
	\end{aligned}
\end{equation*}

We now calculate the integrals:
\begin{lemma}\label{UE2a}
	For any $p \in \M$,
	\begin{equation*}
		\begin{aligned}
			\int_{\M}U_{*}^{p+1}dv_g
			&=\eps^{N}\Big[\int_{\rN}U^{p+1}\(y\)dy-\dfrac{\eps^{2}}{6}\dfrac{S(p)}{N}\int_{\rN}U^{p+1}\(y\)|y|^{2}dy+o(\eps^{2})\Big],\\
			\int_{\M}V_{*}^{q+1}dv_g
			&=\eps^{N}\Big[\int_{\rN}V^{q+1}\(y\)dy- \dfrac{\eps^{2}}{6}\dfrac{S(p)}{N}\int_{\rN}V^{q+1}\(y\)|y|^{2}dy+o(\eps^{2})\Big],
		\end{aligned}
	\end{equation*}
	as $\eps \to 0$, where $S(p)$ is the Scalar curvature of $\M$ at $p \in \M$.
\end{lemma}

\begin{proof}
	We first note that from the decay estimate \eqref{e.decay}, we have,
	\begin{equation*}
		\begin{aligned}
			\left|\int_{B(0, \tfrac{R}{\eps}) \setminus B(0, \tfrac{R}{2\eps})}U^{p+1}\(y\)\sqrt{g(\eps y)} dy\right|
			&\leq \int_{B(0, \tfrac{R}{\eps})\setminus B(0, \tfrac{R}{2\eps})}|U(y)|^{p+1}\left|1+O((\eps |y|)^{2})\right|  dy\\
			&\leq \int_{ \tfrac{R}{2\eps}}^{\infty}Ce^{-c(p+1)\delta r}r^{N-1}\(1+O((\eps r)^{2})\)  dr\\
			&\leq C\(\dfrac{R}{2\eps}\)^{N+2}e^{-\tfrac{c(p+1)R}{\eps}}.
		\end{aligned}
	\end{equation*}
	This implies
	\begin{equation*}
		\begin{aligned}
			\int_{\M}U_{*}^{p+1}dv_g
			&=\int_{B_{g}(p,R)}\phi_{R}^{p+1}\(\Psi_{p}(x)\)U^{p+1}_{\eps}\(\Psi_{p}(x)\)dv_g\\
			&=\int_{B(0, R)}\phi_{R}^{p+1}\(x\)U^{p+1}\(\dfrac{x}{\eps}\)\sqrt{g_{p}(x)} dx\\
			&=\eps^{N}\Big[\int_{B(0, \tfrac{R}{2\eps})}U^{p+1}\(y\)\sqrt{g_{p}(\eps y)} dy+\int_{B(0, \tfrac{R}{\eps})\setminus B(0, \tfrac{R}{2\eps})}\phi_{R}^{p+1}\(\eps y\)U^{p+1}\(y\)\sqrt{g_{p}(\eps y)} dy\Big]\\
			&=\eps^{N}\Big[\int_{\rN}U^{p+1}\(y\)\sqrt{g_{p}(\eps y)} dy-\int_{\rN\setminus B(0, \tfrac{R}{2\eps})}U^{p+1}\(y\)\sqrt{g_{p}(\eps y)} dy+O(exp\(\tfrac{-c}{\eps}\)\Big]\\
			&=\eps^{N}\Big[\int_{\rN}U^{p+1}\(y\)dy-\int_{\rN}\dfrac{1}{6}Ric(\eps y, \eps y)U^{p+1}\(y\)dy+o(\eps^{2})\Big]\\
			&=\eps^{N}\Big[\int_{\rN}U^{p+1}\(y\)dy-\dfrac{\eps^{2}}{6}\dfrac{S(p)}{N}\int_{\rN}U^{p+1}\(y\)|y|^{2}dy+o(\eps^{2})\Big].\\
		\end{aligned}
	\end{equation*}
	Similarly, we can calculate for $V_{*}$ and
	\begin{equation*}
		\int_{\M}V_{*}^{q+1}dv_g =\eps^{N}\Big[\int_{\rN}V^{q+1}\(y\)dy-\dfrac{\eps^{2}}{6}\dfrac{S(p)}{N}\int_{\rN}V^{q+1}\(y\)|y|^{2}dy+o(\eps^{2})\Big].
	\end{equation*}
\end{proof}

We now estimate the error in $h_{\eps}(t)$.
\begin{lemma}\label{UE1a}
	We have
	\begin{equation*}
		\int_{\M}\left( U_{*}^{p}\mathit{T_{1}}V_{*}^{q}+V_{*}^{q}\mathit{T_{2}}U_{*}^{p}\right) dv_g=\int_{\M} U_{*}^{p+1}+V_{*}^{q+1}dv_g+ o(\eps^{N+2}),
	\end{equation*}
	as $\eps \to 0$.
\end{lemma}

\begin{proof}
	We divide the proof into two steps:\\
	\textbf{Step 1:}
	We first estimate the following error terms. Let
	\begin{equation}\label{t1t2}
		\begin{aligned}
			\mathit{T_{1}}V_{*}^{q}&=U_{*}+ \eta_{\eps}, \quad \text{ and } \quad
			\mathit{T_{2}}U_{*}^{p}&=V_{*} + \xi_{\eps} \quad \text{ in } \M.
		\end{aligned}
	\end{equation}
	Solving for $\eta_{\eps}$, we have
	\begin{equation*}
		\begin{aligned}
			\left( -\eps^{2}\Dm +\mathit{id}\right)\eta_{\eps}(x)&=V_{*}^{q}(x)-\left( -\eps^{2}\Dm +\mathit{id}\right)U_{*}(x) \quad \text{ in } \M.
		\end{aligned}
	\end{equation*}
	In local coordinates, using \eqref{LapBel1}, we get
	\begin{equation*}
		\Dm\left( \phi_{R}\(\Psi_{p}(x)\)U_{\eps}\(\Psi_{p}(x)\)\right)=\left(\Delta -(g_{p}^{ij}(x)+\delta_{ij})\partial_{ij} +g_{p}^{ij}(x)\Gamma_{ij}^{k}\partial_{k} \right)\left(\phi_{R}(x)U_{\eps}(x)\right).
	\end{equation*}
	This gives,
	\begin{equation*}
		\begin{aligned}
			&\left\Vert V_{*}^{q}(x)-\left( -\eps^{2}\Dm +\mathit{id}\right)U_{*}( x)\right\Vert_{L^{\tfrac{q+1}{q}}(\M)}^{\tfrac{q+1}{q}}\\
			&=  \int_{B_{g}(p_{\eps},R)}\left| \phi_{R}^{q}\(\Psi_{p}(x)\)V^{q}_{\eps}\(\Psi_{p}(x)\)-\left( -\eps^{2}\Dm +\mathit{id}\right)\phi_{R}\(\Psi_{p}(x)\)U_{\eps}\(\Psi_{p}(x)\)\right| ^{\tfrac{q+1}{q}}dv_g\\
			&= \int_{B(0,R)}\left| \phi_{R}^{q}V^{q}_{\eps}-\left( -\eps^{2}\left(\Delta +(g_{p}^{ij}(x)-\delta_{ij})\partial_{ij} -g_{p}^{ij}(x)\Gamma_{ij}^{k}(x)\partial_{k} \right) \(\phi_{R}U_{\eps}\)+\phi_{R}U_{\eps}\right)\right| ^{\tfrac{q+1}{q}}\sqrt{g_{p}(x)} dx\\
			&= \int_{B(0,R)}\left| \phi_{R}^{q}V^{q}_{\eps}-\(-\eps^{2}\Delta \(\phi_{R}U_{\eps}\)+\phi_{R}U_{\eps}\)+ \eps^{2}\left((g_{p}^{ij}(x)-\delta_{ij})\partial_{ij} -g_{p}^{ij}(x)\Gamma_{ij}^{k}(x)\partial_{k} \right) \(\phi_{R}U_{\eps}\) \right| ^{\tfrac{q+1}{q}}\sqrt{g_{p}(x)} dx\\
			&= \eps^N\int_{B(0,R/\eps)}| \phi_{R}^{q}(\eps y)V^{q}(y)-(-\Delta +id)\(\phi_{R}(\eps y)U(y)\)+(g_{p}^{ij}(\eps y)-\delta_{ij})\partial_{ij}\(\phi_{R}(\eps y)U(y)\)\\
			&\hspace{2cm}-\eps g_{p}^{ij}(\eps y)\Gamma_{ij}^{k}(\eps y)\partial_{k}\(\phi_{R}(\eps y)U(y)\) | ^{\tfrac{q+1}{q}}\sqrt{g_{p}(\eps y)} dy.
		\end{aligned}
	\end{equation*}
	This implies,
	\begin{equation*}
		\left\Vert V_{*}^{q}(x)-\left( -\eps^{2}\Dm +\mathit{id}\right)U_{*}( x)\right\Vert_{L^{\tfrac{q+1}{q}}(\M)} \le  I_1+I_2+I_3,
	\end{equation*}
	where 
	\begin{equation*}
		\begin{aligned}
			I_1:&=\eps^{\frac{Nq}{q+1}} \Big[\int_{B(0,R/\eps)} \abs{\phi_{R}^{q}(\eps y)V^{q}(y)-[(-\Delta+id) \(\phi_{R}(\eps y)U(y\))]}^{\tfrac{q+1}{q}}\sqrt{g_{p}(\eps y)} dy\Big]^{\frac{q}{q+1}},\\
			I_2:&=\eps^{\frac{Nq}{q+1}} \Big[\int_{B(0,R/\eps)} \abs{\left((g_{p}^{ij}(\eps y)-\delta_{ij})\partial_{ij} \right)\(\phi_{R}(\eps y)U(y)\)}^{\tfrac{q+1}{q}}\sqrt{g_{p}(\eps y)} dy\Big]^{\frac{q}{q+1}},\\
			I_3:&=\eps^{\frac{Nq}{q+1}+1} \Big[\int_{B(0,R/\eps)} \abs{\left(g_{p}^{ij}(\eps y)\Gamma_{ij}^{k}(\eps y)\partial_{k} \right)\(\phi_{R}(\eps y)U(y)\)}^{\tfrac{q+1}{q}}\sqrt{g_{p}(\eps y)} dy\Big]^{\frac{q}{q+1}}.
		\end{aligned}
	\end{equation*}
	\textbf{Estimate for $I_1$, $I_{2}$ and $I_{3}$:} 
	We have that
	\begin{align*}
		&\phi_{R}^{q}(\eps y)V^{q}(y)+\Delta \(\phi_{R}(\eps y)U(y\))-\phi_{R}(\eps y)U(y)\\
		&=\phi_{R}^{q}(\eps y)V^{q}(y)+\De U(y)\phi_R(\eps y)+U(y) \De \phi_R(\eps y) +2\grad U(y))\grad \phi_R(\eps y)-\phi_{R}(\eps y)U(y)\\
		&=\phi_{R}^{q}(\eps y)V^{q}(y)-\(-\De U(y)+U(y)\)\phi_R(\eps y)+U(y) \De \phi_R(\eps y) +2\grad U(y)\grad \phi_R(\eps y)\\
		&=[\phi_{R}^{q}(\eps y)-\phi_{R}(\eps y)]V^{q}(y)+U(y) \De \phi_R(\eps y)+2\grad U(y)\grad \phi_R(\eps y).
	\end{align*}
	Now from the exponential decay of $(U, V)$, we conclude that
	\begin{equation*}
		I_1=O(exp\left(-\tfrac{c}{\eps}\right)), \,\text{for some } c>0.
	\end{equation*}
	And using \eqref{metricexpansion}, we have that
	\begin{equation*}
		I_2:=\eps^{\frac{Nq}{q+1}} \Big[\int_{B(0,R/\eps)} \abs{(g_{p}^{ij}(\eps y)-\delta_{ij})\partial_{ij} \(\phi_{R}(\eps y)U(y)\)}^{\tfrac{q+1}{q}}\sqrt{g_{p}(\eps y)} dy\Big]^{\tfrac{q}{q+1}}=O(\eps^{\frac{Nq}{q+1}+2}).
	\end{equation*}
	Similarly, we use \eqref{eqChristoffel} to get
	\be I_3:=\eps^{\frac{Nq}{q+1}+1} \Big[\int_{B(0,R/\eps)} \abs{\left(g_{p}^{ij}(\eps y)\Gamma_{ij}^{k}(\eps y)\partial_{k} \right)\(\phi_{R}(\eps y)U(y)\)}^{\tfrac{q+1}{q}}\sqrt{g_{p}(\eps y)} dy\Big]^{\frac{q}{q+1}}=O(\eps^{\frac{Nq}{q+1}+2}).\nee
	From Sobolev Embedding and $L_{p}$-estimate, we get the estimate
	\begin{equation*}
		\begin{aligned}
			\left\Vert\eta_{\eps}\right\Vert_{L^{p+1}(\M)}\leq c\left\Vert \eta_{\eps}\right\Vert_{W^{2,\tfrac{q+1}{q}}(\M)}\leq c\left\Vert V_{*}^{q}(x)-\left( -\eps^{2}\Dm +\mathit{id}\right)U_{*}( x)\right\Vert_{L^{\tfrac{q+1}{q}}(\M)}=O(\eps^{\tfrac{Nq}{q+1}+2}).
		\end{aligned}
	\end{equation*} 
	Similarly we can obtain an estimate for $\xi_{\eps}$, i.e.,
	\begin{equation*}
		\left\Vert\xi_{\eps}\right\Vert_{L^{q+1}(\M)}\leq c\left\Vert\xi_{\eps}\right\Vert_{W^{2,\frac{p+1}{p}}(\M)}=O(\eps^{\tfrac{Np}{p+1}+2}).
	\end{equation*}
	\textbf{Step 2:  }
	Using \eqref{t1t2} we get
	\begin{equation*}
		\int_{\M}\left( U_{*}^{p}\mathit{T_{1}}V_{*}^{q}+V_{*}^{q}\mathit{T_{2}}U_{*}^{p}\right) dv_g=\int_{\M} \left( U_{*}^{p+1}+V_{*}^{q+1}\right)  dv_g+\int_{\M}\left( U_{*}^{p} \eta_{\eps}+V_{*}^{q} \xi_{\eps}\right) dv_g.
	\end{equation*}
	Now from \cref{UE2a}, we have
	\begin{equation*}
		\begin{aligned}
			\int_{\M} \abs{U_{*}^{p} \eta_{\eps}}dv_g\leq\|U_{*}\|_{L^{p+1}(\M)}^{p}\|\eta_{\eps}\|_{L^{p+1}(\M)}=O(\eps^{\tfrac{Np}{p+1}+\tfrac{Nq}{q+1}+2}).
		\end{aligned}
	\end{equation*}
	We see that
	\begin{equation*}
		\begin{aligned}
			\dfrac{Np}{p+1}+\dfrac{Nq}{q+1}=N+N-\dfrac{N}{p+1}-\dfrac{N}{q+1}>N,
		\end{aligned}
	\end{equation*}
	since
	\begin{equation*}
		\begin{aligned}
			N- \dfrac{N}{p+1}-\dfrac{N}{q+1}>0 \quad \text{ for } pq>1.
		\end{aligned}
	\end{equation*}
	Using the above estimates, we get
	\begin{equation*}
		\left| \int_{\M} U_{*}^{p}\eta_{\eps}\right| =o(\eps^{N+2}), \quad \text{ and } \left| \int_{\M} V_{*}^{q}\xi_{\eps}\right|  =o(\eps^{N+2}).
	\end{equation*}	
\end{proof}

\begin{lemma}[Maximum Point of the MP curve]\label{maxpta}
	For each $\eps>0$ sufficiently small, $h_{\eps}(t):=I_{\eps}(tw)$ attains a unique positive maximum at $t=t^*_\eps>0$ and
	\begin{equation}\label{MaxPta}
		t^*_\eps=1+a\eps^{2}+o(\eps^{2}),
	\end{equation}
	as $\eps \to 0$, for some constant $a$.
\end{lemma}
\begin{proof}
	First we show that $h_{\eps}(t)$ has a unique positive maximum. Differentiating with respect to $t$ we have,
	\begin{equation*}
		\begin{aligned}
			h_{\eps}^{\prime}(t)&=t^{\tfrac{1}{p}}\int_{\M}U_{*}^{p+1}(x)dv_g+t^{\tfrac{1}{q}}\int_{\M}V_{*}^{q+1}(x)dv_g
			-t\int_{\M}\left( U_{*}^{p}\mathit{T_{1}}V_{*}^{q}+V_{*}^{q}\mathit{T_{2}}U_{*}^{p}\right) dv_g.
		\end{aligned}
	\end{equation*}
	We first note that $h_{\eps}^{\prime}(0)=0$, $h_{\eps}^{\prime}(t)>0$ for $t$ small and $h_{\eps}^{\prime}(t)\to-\infty$ as $t\to\infty$.
	Also if we look at the second derivative
	\begin{equation*}
		\begin{aligned}
			h''_{\eps}(t)&=\frac{t^{\tfrac{1}{p}-1}}{p}\int_{\M}U_{*}^{p+1}(x)dv_g+\frac{t^{\tfrac{1}{q}-1}}{q}\int_{\M}V_{*}^{q+1}(x)dv_g
			-\int_{\M}\left( U_{*}^{p}\mathit{T_{1}}V_{*}^{q}+V_{*}^{q}\mathit{T_{2}}U_{*}^{p}\right) dv_g,
		\end{aligned}
	\end{equation*}
	is a decreasing function in $t$. Hence $h_\eps(t)$ has a unique positive maximum say at $t=t^*_\eps.$ Let us now define $\sigma(\eps,t):=\eps^{N}h_{\eps}^{\prime}(t)$. Then from \cref{UE2a} and \cref{UE1a} we have
	\begin{align}\label{maxptt}
		\sigma(\eps, t)&=t^{\tfrac{1}{p}}\left( A_{0}+\eps^{2} A_{1}+o(\eps^{2})\right)+ t^{\tfrac{1}{q}}\left( B_{0}+\eps^{2} B_{1}+o(\eps^{2})\right)
		-t\left( A_{0}+B_{0}+\eps^{2} (A_{1}+B_{1})+o(\eps^{2})\right)
		\notag\\
		&=\big(t^{\tfrac{1}{p}}-t\big)A_{0}+\big(t^{\tfrac{1}{q}}-t\big) B_{0}+\eps^{2}\(\big(t^{\tfrac{1}{p}}-t\big)A_{1}+\big(t^{\tfrac{1}{q}}-t\big)B_{1}\)+ o(\eps^{2}),
	\end{align}
	where
	\begin{align*}
		&A_{0}:=\int_{\rN}U^{p+1}\(y\)dy, \qquad && B_0:=\int_{\rN}V^{q+1}\(y\)dy,\\
		&A_1:=\dfrac{1}{6}\dfrac{S(p)}{N}\int_{\rN}U^{p+1}\(y\)|y|^{2}dy,\qquad && B_1:= \dfrac{1}{6}\dfrac{S(p)}{N}\int_{\rN}V^{q+1}\(y\)|y|^{2}dy.
	\end{align*}
	In this expression it is worth noting that, $o(\eps)$ is uniform in $t$ on each compact interval, which allows us to extend $\sigma(\eps, t)$ up to $\eps=0$ as a continuously differentiable function on $[{0,\eps_{*}})\times [{0,\infty})$ for some $\eps_{*}>0$. Now we see here, 
	\begin{equation*}
		\sigma(0, 1)=0, \qquad \sigma_\eps(0, 1)=0, \qquad \sigma_t(0, 1)=\left(\frac{1}{p}-1 \right)A_{0}+ \left(\frac{1}{q}-1 \right)B_{0} <0.
	\end{equation*}
	Then by the Implicit Function Theorem we have a $C^{1}$ function $t(\eps)$ defined for $\eps \in [{0, \eps_{**}})$ with $\eps_{**}>0$ sufficiently small such that $\sigma(\eps, t(\eps))=0$ and $t(0)=1$. Now expanding $\xi(\eps):=\sigma(\eps, t(\eps))$ around $0$ we get
	\be \xi(\eps)=\xi(0)+\eps(\sigma_\eps(0,1)+\sigma_t(0,1)t'(0))+\eps^2\xi''(0)+o(\eps^2).\nee This gives us \be t'(0)=c\eps+o(\eps),\nee for some constant $c$ as $\sigma_t(0,1)\neq 0$. And finally, using Taylor's expansion for $t(\eps)$ around $0$, we get the expression \eqref{MaxPta} for the maximum point.
\end{proof}

Finally, we have the following upper energy estimate:
\begin{prop}[Upper Energy Estimate]\label{UEEa}
	\begin{equation*}
		c_{\eps} \leq \eps^{N}\left[ I_{\infty}(U^{p}, V^{q})-\eps^{2}\dfrac{S(p)}{6N}\(\dfrac{p-1}{2(p+1)}\int_{\rN}U^{p+1}(y)|y|^{2}dy-\dfrac{q-1}{2(q+1)}\int_{\rN}V^{q+1}\(y\)|y|^{2}dy\)+ o(\eps^{2})\right],	
	\end{equation*}
	where $S(p)$ is the Scalar curvature of $\M$ at $p \in \M$.
\end{prop}

\begin{proof}
	Let $\gamma(t):=tw.$ Then we have
	\begin{equation*}
		\begin{aligned}
			I_{\eps}(\ga(t))&=\dfrac{pt^{\tfrac{p+1}{p}}}{p+1}\int_{\M}U_{*}^{p+1}(x)dv_g+\dfrac{qt^{\tfrac{q+1}{q}}}{q+1}\int_{\M}V_{*}^{q+1}(x)dv_g-\dfrac{t^{2}}{2}\int_{\M}\left( U_{*}^{p}\mathit{T_{1}}V_{*}^{q}+V_{*}^{q}\mathit{T_{2}}U_{*}^{p}\right) dv_g.
		\end{aligned}
	\end{equation*}
	From \cref{UE2a} and \cref{UE1a}, we note that $I_{\eps}(\ga(1))>0$ as $p,q>1$. And there exist $T_0>1$ such that $	I_{\eps}(\ga(T_0))<0$. On redefining $\ga(t)$ as $\ga(tT_0)$, we see that $I_{\eps}(\ga(1))<0$. From \eqref{lstenergy} and \cref{maxpta}, we get
	\begin{equation*}
		c_{\eps} \leq \max_{0 \leq t \leq T_0}I_{\eps}(tw)=I_{\eps}(t_{\eps}^{*}w).
	\end{equation*}
	Now let us calculate $I_\eps(t_{\eps}^{*}w)$.
	\begin{equation*}
		\begin{aligned}
			I_{\eps}(t^*_\eps w_{\eps})=&\left( \dfrac{p}{p+1}{t^*_\eps}^{\tfrac{p+1}{p}}-\dfrac{{t^*_\eps}^{2}}{2}\right) \int_{\M}U_{*}^{p+1}dv_g+\left( \dfrac{q}{q+1}{t^*_\eps}^{\tfrac{q+1}{q}}-\dfrac{{t^*_\eps}^{2}}{2}\right)\int_{\M}V_{*}^{q+1}dv_g+o(\eps^{N+2}).
		\end{aligned}
	\end{equation*}
	We have
	\begin{equation}\label{Maxpt1}
		\begin{aligned}
			\dfrac{p}{p+1}{t^*_\eps}^{\tfrac{p+1}{p}}-\dfrac{{t^*_\eps}^{2}}{2}&=\dfrac{p}{p+1}\left( 1+\dfrac{p+1}{p}a\eps^{2}+o(\eps^{2})\right) -\dfrac{1}{2}\left(1+2a\eps^{2}+o(\eps^{2})\right)\\
			&=\dfrac{p-1}{2(p+1)}+o(\eps^{2}).\\
			\text{Similarly, } \quad \dfrac{q}{q+1}{t^*_\eps}^{\tfrac{q+1}{q}}-\dfrac{{t^*_\eps}^{2}}{2}&=\dfrac{q-1}{2(q+1)}+o(\eps^{2}).
		\end{aligned}
	\end{equation}
	Using the expansions in \cref{UE2a} we finally get
	\begin{equation*}
		\begin{aligned}
			I_{\eps}(t^*_\eps  w_{\eps})
			&=\eps^{N}\left[ \left(\dfrac{p-1}{2(p+1)}\right) \int_{\rN}U^{p+1}\(y\)dy+\left(\dfrac{q-1}{2(q+1)}\right) \int_{\rN}V^{q+1}\(y\)dy\right.\\ &\left.-\eps^{2}\dfrac{S(p)}{6N}\(\dfrac{p-1}{2(p+1)}\int_{\rN}U^{p+1}(y)|y|^{2}dy-\dfrac{q-1}{2(q+1)} \int_{\rN}V^{q+1}(y)|y|^{2}dy\)+ o(\eps^{2})\right]\\
			&=\eps^{N}\left[ I_{\infty}(U^{p}, V^{q})-\eps^{2}\dfrac{S(p)}{6N}\(\dfrac{p-1}{2(p+1)}\int_{\rN}U^{p+1}(y)|y|^{2}dy-\dfrac{q-1}{2(q+1)} \int_{\rN}V^{q+1}(y)|y|^{2}dy\)+ o(\eps^{2})\right],
		\end{aligned}
	\end{equation*}
	which gives us the upper energy estimate.
\end{proof}

\section{Asymptotic Profile of the Least energy solution}

In this section, we establish several key results concerning the least energy solutions associated with the energy level $c_\eps$. We consider a solution $(u_\eps,v_\eps)\in W^{2,\beta}(\M)\times W^{2,\al}(\M)$ to \eqref{eq0}, corresponding to a least energy critical point $(w_1,w_2)\in X^*$ of the dual functional $I_\eps$ and investigate several qualitative properties of these solution, namely the number of local maxima and their asymptotic behavior. Remarkably, we demonstrate that for any least energy solution, the local maxima of both $u_{\eps}$ and $v_{\eps}$ occur at a common point. To analyze the asymptotic profile, we rescale the solution around its local maximum and establish $C^2_{loc}$ convergence of the rescaled functions as the perturbation parameter goes to zero.

We begin with the following result, which provides an $L^\infty$ bound for a least energy solution:

\begin{lemma}\label{LinfBound}
	Let $(u_{\eps}, v_{\eps})$ be a least energy solution of $\eqref{eq0}$. Then there exists $\eps_{0}>0$ such that if $\eps \in (0,\eps_{0})$, then
	\begin{equation*}
		\max\{\|u_{\eps}\|_{L^{\infty}}, \|v_{\eps}\|_{L^{\infty}}\} \leq K,
	\end{equation*}
	where $K>0$ is independent of $\eps$.
\end{lemma}

\begin{proof}
	By contradiction, let us suppose that there exists a subsequence $\eps_{j} \to 0$ as $j \to \infty$ such that
	\begin{equation*}
		\max\{\|u_{\eps_{j}}\|_{L^{\infty}}, \|v_{\eps_{j}}\|_{L^{\infty}}\} \to \infty.
	\end{equation*}
	For simplicity, let us denote the sequence $(u_{\eps_{j}}, v_{\eps_{j}})$ by $(u_{j}, v_{j})$. Define two constants, $\beta_{1}, \beta_{2}$ as
	\begin{equation*}
		\beta_{1}=\dfrac{2(1+q)}{pq-1}, \qquad \beta_{2}=\dfrac{2(1+p)}{pq-1}.
	\end{equation*}
	Without loss of generality we assume 
	\begin{equation*}
		u_{j}(p_{j})= \|u_{\eps_{j}}\|_{L^{\infty}(\M)} \to \infty \text{ and } \|v_{j}\|_{L^{\infty}(\M)} \leq \|u_{j}\|_{L^{\infty}(\M)},
	\end{equation*}
	for some $p_j\in\M$. Let $\lambda_{j}$ be a sequence of real numbers with $\lambda_{j}^{\beta_{1}}\|u_{j}\|_{L^{\infty}} =1$. Then $\lambda_{j} \to 0$ as $j \to \infty$. For any $R<R^*$, we take the radial cutoff function $\Phi_R$ as in \eqref{radcutoff} and define the scaled functions $\tilde{u}_{j}$ and $\tilde{v}_{j}$ in $B(0,R/2\eps_{j}\lambda_{j})$ as:
	\begin{equation*}
		\tilde{u}_{j}(x) = \lambda_{j}^{\beta_{1}}u_{j}\(\Psi_{p_{j}}^{-1}( \eps_{j} \lambda_{j}x)\)\Phi_R\( \eps_{j} \lambda_{j}x\),  \qquad \tilde{v}_{j}(x) = \lambda_{j}^{\beta_{2}}v_{j}\(\Psi_{p_{j}}^{-1}( \eps_{j} \lambda_{j}x)\)\Phi_R\( \eps_{j} \lambda_{j}x\). 
	\end{equation*}
	Then we have $\|\tilde{u}_{j}\|_{L^{\infty}} \leq 1$, and $\tilde{u}_{j}(0)= \lambda_{j}^{\beta_{1}} u_{j}(p_{j})=1$. Also, $(\tilde{u}_{j}, \tilde{v}_{j})$ satisfy the system
	\begin{equation*}
		\begin{aligned}
			-\dfrac{1}{\sqrt{g_{p}(\eps_{j} \lambda_{j} x)}}\sum_{i, j}\partial_{i}\(g^{ij}_{p}(\eps_{j} \lambda_{j} x)\sqrt{g_{p}(\eps_{j} \lambda_{j} x)}\partial_{j}\tilde{u}_{j}\) +\lambda_{j}^{2}\tilde{u}_{j}&=\lambda_{j}^{\beta_{1}+2-\beta_{2}q}\tilde{v}^{q}_{j}=\tilde{v}^{q}_{j},\\
			-\dfrac{1}{\sqrt{g_{p}(\eps_{j} \lambda_{j} x)}}\sum_{i, j}\partial_{i}\(g^{ij}_{p}(\eps_{j} \lambda_{j} x)\sqrt{g_{p}(\eps_{j} \lambda_{j} x)}\partial_{j}\tilde{v}_{j}\) +\lambda_{j}^{2}\tilde{v}_{j}&=\lambda_{j}^{\beta_{2}+2-\beta_{1}p}\tilde{u}^{p}_{j}=\tilde{u}^{p}_{j}.
		\end{aligned}
	\end{equation*}
	We can easily see that $\beta_{1}+2-\beta_{2}q=0=\beta_{2}+2-\beta_{1}p$. From \cref{UEEa} we have
	\be c_\eps\le \eps^N\ml{C}_\infty.\nee
	Since $(u_\eps,v_\eps)$ is a least energy solution we get
	\be \Big[\f{1}{2}-\f{1}{p+1}\Big]\int_{\M}|u_\eps|^{p+1}dv_g+\Big[\f{1}{2}-\f{1}{q+1}\Big]\int_{\M}|v_\eps|^{q+1}dv_g<C \text{ uniformly } \forall \eps<<1.\nee
	Thus, for some $C>0$,
	\begin{align*}
		C\eps_j^N&\ge \int_{\M}u_j^{p+1}dv_g\ge \int_{B_g(p_j,R/2)}u_j^{p+1}dv_g\\
		&=\int_{B(0,R/2\la_j\eps_j)}\abs{\la_j}^{-\beta_1(p+1)}\tilde{u}_j^{p+1}(z)\sqrt{g(\la_j\eps_j z)}d(\la_j\eps_j z)\\
		&\ge C_1\eps_j^N\int_{B(0,R/2\la_j\eps_j)}\abs{\la_j}^{N-\beta_1(p+1)}\tilde{u}_j^{p+1}(z)dz.
	\end{align*}
	Now from \ef{HC} we note that $N-\beta_1(p+1)=N-\frac{2(p+1)(q+1)}{pq-1}< 0.$ Hence, as $\la_j\to 0$ we have (with a similar analysis)
	\be \int_{B(0,R/2\la_j\eps_j)}\tilde{u}_j^{p+1}(z)dz, \,\text{ and } \int_{B(0,R/2\la_j\eps_j)}\tilde{v}_j^{q+1}(z)dz<C,\nee
	uniformly for any $0<\eps_j<<1.$ Using bootstrap arguments and Sobolev embedding (as in the \cref{APPB}) $\(\tilde{u}_{j}, \tilde{v}_{j}\)$ is bounded in $C^{2,\alpha}\(B(0,r)\)\times C^{2,\alpha}\(B(0,r)\)$ for some $\alpha \in (0,1)$ and any $r<R/2\la_j\eps_j$. Hence, there exists a constant $C>0$, independent of $j$, such that
	\begin{equation*}
		\|\tilde{u}_{j}\|_{C^{2,\alpha}\(B(0,r)\)}+\|\tilde{v}_{j}\|_{C^{2,\alpha}\(B(0,r)\)}<C.
	\end{equation*}
	So the functions $\tilde{u}_{j}$ and $\tilde{v}_{j}$ and their derivatives up to order $2$ are uniformly bounded on $B(0,r)$. Also, from the definition of H\"{o}lder continuity, we have
	\begin{equation*}
		|D^{k}\tilde{u}_{j}(x)-D^{k}\tilde{u}_{j}(y)| \leq C|x-y|^{\alpha}, \qquad \text{ for all } |k|=2, \; x, y \in B(0,r), \; x\neq y , \text{ and for all $j$ large}.
	\end{equation*}
	A similar estimate holds for $\tilde{v}_{j}$. From the uniform bound, we also have that $\tilde{u}_{j}$ and $\tilde{v}_{j}$ and their derivatives up to order $2$ are equicontinuous and uniformly bounded. By Arzela-Ascoli theorem, we conclude that up to a subsequence,
	\begin{equation*}
		(\tilde{u}_{j}, \tilde{v}_{j}) \to (\tilde{u}, \tilde{v})\quad \text{ in } \quad C^{2}(B(0,r)) \times C^{2}(B(0, r)).
	\end{equation*}
	Since $r>0$ was arbitrary and $r \to \infty$, a standard diagonal argument gives
	\begin{equation*}
		(\tilde{u}_{j}, \tilde{v}_{j}) \to (\tilde{u}, \tilde{v})\quad \text{ in } \quad C^{2}_{loc}(\rN) \times C^{2}_{loc}(\rN),
	\end{equation*}
	where $(\tilde{u}, \tilde{v})$ satisfy the system
	\begin{equation*}
		-\Delta\tilde{u}=\tilde{v}^{q}, \qquad
		-\Delta\tilde{v}=\tilde{u}^{p} \quad \text{ in } \R^{N}.
	\end{equation*}
	Since $\(\tilde{u}_{j}, \tilde{v}_{j}\)$ are positive solutions, by \cite{MR1617988}, they are radially symmetric and this implies $(\tilde{u}, \tilde{v})$ are radial positive solutions. Then by Theorem 3.2 of \cite{MR1211727}, $\tilde{u}=\tilde{v}=0$. This is a contradiction since $\tilde{u}_{j}(0)=1$ for every $j$.	
\end{proof}

Now, let $p_\eps, q_\eps\in\M$ be points such that \be \max_{x\in\M} u_\eps=u_\eps(p_\eps) \quad \text{ and } \quad \max_{x\in\M} v_\eps=v_\eps(q_\eps).\nee
As $M$ is compact we find a sequence $\{\eps_j\}\to 0$ and points $p_0, q_0$ in $\M$ such that 
\be p_j:=p_{\eps_j}\to p_0 \quad \text{ and } \quad q_j:=q_{\eps_j}\to q_0.\nee
Then similarly as in the proof of Lemma 4.1 of \cite{MR4599355}, we can show that \be u_{\eps}(p_{\eps}) \geq 1 \qquad \text{ and } \quad v_{\eps}(q_{\eps})\geq 1.\label{lbd.maxpt}\ee 

Now we shall analyze the behavior of the points of maximum for the solution $(u_\eps,v_\eps).$ It is very crucial to notice here that from the \cref{UEEa} we have \be \limsup_{\eps\to 0} \frac{c_\eps}{\eps^N}\le \ml{C}_\infty.\label{unq.max.pt1}\ee

\begin{lemma}\label{ComMaxPt1}
	There exists $\eps_{0}>0$ such that for all $\eps \in (0,\eps_{0})$, any global maximum points $p_{\eps}$ and $q_{\eps}$ of $u_{\eps}$ and $v_{\eps}$, respectively, satisfies:
	\begin{equation}\label{distmaxpts}
		\frac{d_g(p_\eps,q_\eps)}{\eps} \to 0, \qquad \text{ as } \qquad \eps \to 0.
	\end{equation}
	Moreover, the sequences $\{p_{\eps}\}$ and $\{q_{\eps}\}$ converge to a unique subsequential limit $p_{0} \in \M$. 
\end{lemma}

\begin{proof}
	Take any sequence $\eps_j\to 0$ and denote $p_{\eps_{j}}$ and $q_{\eps_{j}}$ by $p_{j}$ and $q_{j}$. For $R$ being the injectivity radius of $\M$, we define for $|\eps_{j} z|<\frac{R}{2}$ the following,
	\begin{equation*}
		\begin{cases}
			\begin{aligned}
				\tilde{u}_{j}(z) &=u_{j}\(\Psi_{p_{j}}^{-1}( \eps_{j} z)\) ,\\
				\tilde{v}_{j}(z) &=v_{j}\(\Psi_{p_{j}}^{-1}( \eps_{j} z)\).
			\end{aligned}
		\end{cases} \text{ and } \quad 
		\begin{cases}
			\begin{aligned}
				\bar{u}_{j}(z) &=u_{j}\(\Psi_{q_{j}}^{-1}( \eps_{j} z)\),\\
				\bar{v}_{j}(z) &=v_{j}\(\Psi_{q_{j}}^{-1}( \eps_{j} z)\).
			\end{aligned}
		\end{cases}
	\end{equation*}
	Using boot-strap arguments as in the \cref{APPB} we can show $(\tilde{u}_{j},\tilde{v}_{j}) \to (\tilde{u}, \tilde{v})$ and $(\bar{u}_{j},\bar{v}_{j}) \to (\bar{u}, \bar{v})$ in $C^{2}_{loc}(\rN)$, where $(\tilde{u}, \tilde{v})$ and $(\bar{u}, \bar{v})$ solves $\eqref{es}$. Now by contradiction, let $\eps_{j}R\leq \dfrac{d_{g}(p_{j},q_{j})}{2}$, which means $B(p_{j}, \eps_{j}R)\cap B(q_{j}, \eps_{j}R)=\emptyset$. Then we see that
	\begin{equation*}
		\begin{aligned}
			\Gamma_{\eps}(u_{j}, v_{j})
			&=\(\dfrac{1}{2}-\dfrac{1}{p+1}\)\int_{\M}u_{j}^{p+1}dv_{g}+\(\dfrac{1}{2}-\dfrac{1}{q+1}\)\int_{\M}v_{j}^{q+1}dv_{g}\\
			&\geq\(\dfrac{1}{2}-\dfrac{1}{p+1}\)\int_{B_{g}(p_{j},\eps_{j}R)}u_{j}^{p+1}dv_{g}+\(\dfrac{1}{2}-\dfrac{1}{q+1}\)\int_{B_{g}(p_{j},\eps_{j}R)}v_{j}^{q+1}dv_{g}\\
			&+\(\dfrac{1}{2}-\dfrac{1}{p+1}\)\int_{B_{g}(q_{j},\eps_{j}R)}u_{j}^{p+1}dv_{g}+\(\dfrac{1}{2}-\dfrac{1}{q+1}\)\int_{B_{g}(q_{j},\eps_{j}R)}v_{j}^{q+1}dv_{g}\\
			&=\eps_{j}^{N}\left[ \(\dfrac{1}{2}-\dfrac{1}{p+1}\)\int_{B(0,R)}\tilde{u}_{j}^{p+1}(z)dz+\(\dfrac{1}{2}-\dfrac{1}{q+1}\)\int_{B(0,R)}\tilde{v}_{j}^{q+1}(z)dz\right.\\
			&+\left.\(\dfrac{1}{2}-\dfrac{1}{p+1}\)\int_{B(0,R)}\bar{u}_{j}^{p+1}(z)dz+\(\dfrac{1}{2}-\dfrac{1}{q+1}\)\int_{B(0,R)}\bar{v}_{j}^{q+1}(z)dz+o(1)\right].
		\end{aligned}
	\end{equation*}
	Now from the $C^{2}_{loc}$ convergence and the exponential decay of limits, we obtain
	\begin{equation*}
		\begin{aligned}
			\dfrac{\Gamma_{\eps_{j}}(u_{j}, v_{j})}{\eps_{j}^{N}}
			&\geq\(\dfrac{1}{2}-\dfrac{1}{p+1}\)\int_{\rN}\tilde{u}^{p+1}(z)dz+\(\dfrac{1}{2}-\dfrac{1}{q+1}\)\int_{\rN}\tilde{v}^{q+1}(z)dz\\
			&+\(\dfrac{1}{2}-\dfrac{1}{p+1}\)\int_{\rN}\bar{u}^{p+1}(z)dz+\(\dfrac{1}{2}-\dfrac{1}{q+1}\)\int_{\rN}\bar{v}^{q+1}(z)dz+o(1)\\
			& \geq \ml{C}_{\infty}+\ml{C}_{\infty}+o(1)\\
			&=2\ml{C}_{\infty}+o(1),
		\end{aligned}
	\end{equation*}
	which is a contradiction. Hence we get $d_{g}\(p_{j},q_{j}\) \to 0$. So, for $j$ large we have $q_j\in B_g(p_j, r)$, for some small $r>0$. Define
	\begin{equation}\label{eqmaxxj}
		x_{j}:=\dfrac{\Psi_{p_{j}}(q_{j})}{\eps_{j}} \in B(0, \tfrac{R}{\eps_{j}}), \quad \text{ so that } \quad \tilde{u}_{j}(0)=u_{j}(\Psi_{p_{j}}(0))=u_{j}(p_{j}), \qquad \tilde{v}_{j}(x_{j})=\tilde{v}_{j}(\tfrac{\Psi_{p_{j} }(q_{j})}{\eps_{j}})=v_{j}(q_{j}).
	\end{equation}
	Thus, $0$ and $x_{j}$ are the respective maximum points of $\tilde{u}_{j}$ and $\tilde{v}_{j}$. Since $(\tilde{u}_{j},\tilde{v}_{j}) \to (\tilde{u}, \tilde{v})$ in $C^{2}_{loc}(\rN) \times C^{2}_{loc}(\rN)$, $\tilde{u}$ and $\tilde{v}$ have only one common maximum point(Theorem 1 of \cite{MR1755067}), we deduce
	\begin{equation*}
		\lim_{j \to \infty}x_{j}= 0, \qquad \text{ and hence } \quad \frac{d_g(p_\eps,q_\eps)}{\eps}\to 0, \qquad \text{ as } \quad \eps \to 0.
	\end{equation*}
	The similar argument as above will conclude the uniqueness of the subsequential limit $p_{0}$ where both $p_{\eps}$ and $q_{\eps}$ concentrate as $\eps \to 0$.
\end{proof}

As $d_g(p_\eps,q_\eps)=o(\eps),$ we shall work around the global maximum point $p_\eps$ of $u_\eps$. For a sequence $\eps_j\to 0$, let us denote $p_{j}:=p_{\eps_j}$ and $q_{j}:=q_{\eps_j}$, with both $p_{j} \to p_{0}$ and $q_{j} \to p_{0}$ as $j \to \infty$.  We now define the functions
\begin{equation}\label{trfsol}
	\begin{cases}
		\begin{aligned}
			U_{j}(z) &=u_{j}\(\Psi_{p_j}^{-1}( \eps_{j} z)\)\Phi_R(\eps_{j} z),\\
			V_{j}(z) &=v_{j}\(\Psi_{p_j}^{-1}( \eps_{j} z)\)\Phi_R(\eps_{j} z)
			\quad \text{ for } z \in B(0,\frac{R}{\eps_{j}}),
		\end{aligned}
	\end{cases}
\end{equation} 
where $\Phi_R$ is the radial cutoff function defined in \ef{radcutoff}, and $U_{j}(z)=V_{j}(z)=0$ in $\rN\backslash B(0,\frac{R}{\eps_{j}}).$ Then from \eqref{lbd.maxpt}, we see that 
\begin{equation*}
	U_j(0)=u_j(p_j)\ge 1 \quad \text{ and } \quad V_j(x_j)=v_j(q_j)\ge 1,
\end{equation*}
where $x_{j}$ is as defined in \eqref{eqmaxxj}.
Using arguments similar to those in \cref{LinfBound}, for any fixed $r<\frac{R}{2\eps_j}$, the sequences
${U}_{j}$ and ${V}_{j}$ are uniformly bounded in $C^{2,\alpha}(B(0,r))$, and up to a subsequence, $(U_j,V_j)\to (U,V)$ in $C^2_{loc}(\rN)$, {where $(U,V)$ is a positive, radial solution of \ef{es}.

\begin{lemma}
	It holds that \be \lim_{\eps\to 0}\eps^{-N}c_\eps=\ml{C}_\infty.\nee 
\end{lemma}

\begin{proof}
	We have from \cref{UEEa} that $\eps^{-N}c_\eps\le \ml{C}_\infty$ for $\eps$ small enough. For the other side inequality we note
	\begin{equation*}
		\begin{aligned}
			\dfrac{\Gamma_{\eps}(u_{j}, v_{j})}{\eps_{j}^{N}}
			&=\dfrac{1}{\eps_{j}^{N}}\left[ \(\dfrac{1}{2}-\dfrac{1}{p+1}\)\int_{\M}u_{j}^{p+1}dv_g+\(\dfrac{1}{2}-\dfrac{1}{q+1}\)\int_{\M}v_{j}^{q+1}dv_g\right] \\
			&\geq \dfrac{1}{\eps_{j}^{N}}\int_{B(0,R/2)}\Big[\(\dfrac{p-1}{2(p+1)}\)u_{j}^{p+1}(\Psi_{p_j}^{-1}(x))+\(\dfrac{q-1}{2(q+1)}\)v_{j}^{q+1}(\Psi_{p_j}^{-1}(x))\Big]\sqrt{g_{p_{j}}(x)}dx\\
			&=\int_{B(0,\tfrac{R}{2\eps_{j}})}\Big[\dfrac{p-1}{2(p+1)}U_{j}^{p+1}(z))+\dfrac{q-1}{2(q+1)}V_{j}^{q+1}(z))\Big]\sqrt{g_{p_{j}}(\eps_{j} z)}dz\\
			&\geq \ml{C}_\infty+o(1).
		\end{aligned}
	\end{equation*}
\end{proof}
Hence we conclude that $(U_j,V_j)\to(U,V)$ in $C^2_{loc}(\rN)$, where $(U,V)$ is a positive, radial, least energy solution of \ef{es}. Also using similar arguments as in the Section 5 of \cite{MR4599355} we can conclude that 
\be \abs{D^\delta U_j(z)}\le Ce^{-\theta |z|} \text{ and } \abs{D^\delta V_j(z)}\le Ce^{-\theta |z|} \, \text{ in } B(0,R/2\eps_j),\label{expoDecayR}\ee
And equivalently
\be \abs{D^\delta u_\eps(p)}\le Ce^{-\theta \frac{d(p,p_\eps)}{\eps}} \text{ and } \abs{D^\delta v_\eps(p)}\le Ce^{-\theta \frac{d(p,p_\eps)}{\eps}}\, \text{ in } B_g(p,R).\label{expoDecayM}\ee
for $|\delta|\le 2$ and some $C, \theta>0$.

\begin{prop}[Lower Energy Estimate]\label{LEEa}
	For small $\eps_{j}$, we have
	\begin{equation*}
		c_{\eps_{j}}\geq\eps_{j}^{N}\left[ \ml{C}_\infty-\eps_{j}^{2}\dfrac{S(p_{j})}{6N}\left(\dfrac{p-1}{2(p+1)}\int_{\R^{N}}U^{p+1}(y)|y|^{2}dy+\dfrac{q-1}{2(q+1)}\int_{\R^{N}}V^{q+1}(y)|y|^{2}dy\right)+o(\eps_{j}^{2})\right],
	\end{equation*}
	where $S(p)$ is the Scalar curvature of $\M$ at $p \in \M$, and $(U,V)$ is a least energy solution of \eqref{es}.
\end{prop}

\begin{proof}
	Let $(u_j,v_j):=(u_{\eps_j},v_{\eps_j})$ be a sequence of least energy solution of \eqref{eq0}, corresponding to the least energy critical points $w_{j}:=(w_{j, 1},w_{j, 2}):=(w_{\eps_j, 1},w_{\eps_j, 2})$ of $I_{\eps_j}$ say. Then from \cref{comp.dual.soln}, we have
	$(w_{j, 1}, w_{j, 2})=(u_{j}^{p}, v_{j}^{q})$  and $ \J_{\eps_j}(u_{j}, v_{j})= I_{\eps_j}\left(w_{j, 1}, w_{j, 2}\right)$. Also,
	
	\be c_{\eps_{j}}=I_{\eps_j}(w_j)\ge I_{\eps_j}(tw_j) \quad \text{ for any } t>0.\label{loenes1}\ee 
	Now 
	\begin{align*}
		I_{\eps_j}(tw_j)&=\int_{\M}\(\dfrac{t^\al}{\al}|w_{j, 1}|^{\al}+\dfrac{t^\beta}{\beta}|w_{j, 2}|^{\beta}\)dv_g -\dfrac{t^2}{2}\int_{\M}\left( w_{j, 1}\mathit{T_{1}}w_{j, 2}+w_{j, 2}\mathit{T_{2}}w_{j, 1}\right) dv_g\\
		&=\int_{B_{g}(p_{j},\tfrac{R}{2})}\(\dfrac{t^\al}{\al}|u_{j}|^{p+1}+\dfrac{t^\beta}{\beta}|v_{j}|^{q+1}\)dv_g -\dfrac{t^2}{2}\int_{B_{g}(p_{j},\tfrac{R}{2})}\left( u_{j}^{p}\mathit{T_{1}}v_{j}^{q}+v_{j}^{q}\mathit{T_{2}}u_{j}^{p}\right) dv_g\\
		&+\int_{\M\backslash B_{g}(p_{j},\tfrac{R}{2})} 	\(\(\dfrac{t^\al}{\al}-\frac{t^2}{2}\)|u_{j}|^{p+1}+\(\dfrac{t^\beta}{\beta}-\frac{t^2}{2}\)|v_{j}|^{q+1}\)dv_g\\
		&=\int_{B_{g}(p_{j},\tfrac{R}{2})}\(\dfrac{t^\al}{\al}|u_{j}|^{p+1}+\dfrac{t^\beta}{\beta}|v_{j}|^{q+1}\)dv_g -\dfrac{t^2}{2}\int_{B_{g}(p_{j},\tfrac{R}{2})}\left( u_{j}^{p}\mathit{T_{1}}v_{j}^{q}+v_{j}^{p}\mathit{T_{2}}u_{j}^{p}\right) dv_g+C\exp(-\frac{c}{\eps_j}).
	\end{align*}
	Hence we have for any $t>0$,
	\be c_{\eps_{j}}\ge H_{j}(t) +C\exp(-\frac{c}{\eps_j}),\label{loenes2}\ee
	where
	\begin{equation*}
		H_{j}(t)=\int_{B_{g}(p_{j},\tfrac{R}{2})}\(\dfrac{t^\al}{\al}|u_{j}|^{p+1}+\dfrac{t^\beta}{\beta}|v_{j}|^{q+1}\)dv_g -\dfrac{t^2}{2}\int_{B_{g}(p_{j},\tfrac{R}{2})}\left( u_{j}^{p}\mathit{T_{1}}v_{j}^{q}+v_{j}^{p}\mathit{T_{2}}u_{j}^{p}\right) dv_g.
	\end{equation*}
	\textbf{Step I: Estimation of the terms:} Using the asymptotic estimate \ef{expoDecayR}, and calculating similarly as in \cref{UE2a} we get,
	\begin{align}
		\int_{B_{g}(p_{j},\tfrac{R}{2})}|u_{j}|^{p+1}dv_g&=\eps_{j}^{N}\Big[\int_{\rN}U_j^{p+1}(y)dy-\int_{\rN}\dfrac{1}{6}Ric(\eps_j y, \eps_j y)U_j^{p+1}(y)dy+o(\eps_{j}^{2})\Big],\label{loenes3}\\
		\text{ and, }\quad \int_{B_{g}(p_{j},\tfrac{R}{2})}|v_{j}|^{q+1}dv_g&=\eps_{j}^{N}\Big[\int_{\rN}V_j^{q+1}(y)dy-\int_{\rN}\dfrac{1}{6}Ric(\eps_j y, \eps_j y)V_j^{q+1}(y)dy+o(\eps_{j}^{2})\Big].\label{loenes4}
	\end{align}
	Now let us calculate the second integral of \ef{loenes2}. To do this let us first calculate the following: 
	\begin{equation*}
		\begin{aligned}
			&\int_{B_{g}(p_{j},\tfrac{R}{2})}u_{j}^{p}T_{1}v_{j}^{q}dv_g\\ &= \int_{B(0,\tfrac{R}{2})}u_{j}^{p}(\Psi_{p_{j}}^{-1}(x))(-\dfrac{\eps^{2}}{\sqrt{g_{p}(x)}}\sum_{i, j}\partial_{i}\(g^{ij}_{p}(x)\sqrt{g_{p}(x)}\partial_{j} \)+id)^{-1}v_{j}^{q}(\Psi_{p_{j}}^{-1}(x))\sqrt{g_{p_{j}}(x)}dx\\
			&= \int_{B(0,\tfrac{R}{2})}U_{j}^{p}\(\dfrac{x}{\eps_{j}}\)\(-\dfrac{\eps^{2}}{\sqrt{g_{p}(x)}}\sum_{i, j}\partial_{i}\(g^{ij}_{p}(x)\sqrt{g_{p}(x)}\partial_{j} \)+id\)^{-1}V_{j}^{q}\(\dfrac{x}{\eps_{j}}\)\sqrt{g_{p_{j}}(x)}dx\\
			&= \int_{B(0,\tfrac{R}{2})}U_{j}^{p}\(\dfrac{x}{\eps_{j}}\)(-\eps^{2}\Delta+id)^{-1}V_{j}^{q}\(\dfrac{x}{\eps_{j}}\)\sqrt{g_{p_{j}}(x)}dx+\int_{B(0,\tfrac{R}{2})}U_{j}^{p}\(\dfrac{x}{\eps_{j}}\)\eta(x)\sqrt{g_{p_{j}}(x)}dx,
		\end{aligned}
	\end{equation*}
	where
	\begin{equation*}
		\eta(x)=\(-\dfrac{\eps^{2}}{\sqrt{g_{p}(x)}}\sum_{i, j}\partial_{i}\(g^{ij}_{p}(x)\sqrt{g_{p}(x)}\partial_{j} \)+id\)^{-1}V_{j}^{q}\(\dfrac{x}{\eps_{j}}\)-(-\eps^{2}\Delta+id)^{-1}V_{j}^{q}\(\dfrac{x}{\eps_{j}}\).
	\end{equation*}
	Let
	\begin{equation*}
		W_{j}\(\dfrac{x}{\eps_{j}}\)=(-\eps^{2}\Delta+id)^{-1}V_{j}^{q}\(\dfrac{x}{\eps_{j}}\).
	\end{equation*}
	Then from \ef{trfsol} we get $\eta(x)=U_{j}(\frac{x}{\eps_{j}}) -W_{j}(\frac{x}{\eps_{j}})$.
	Now note,
	\begin{equation*}
		\begin{aligned}
			\(-\eps^{2}\Delta +id\)U_{j}(\frac{x}{\eps_{j}}) +\eps^{2}(g^{kl}_{p}(x)-\delta_{kl})\partial_{kl}U_{j}(\frac{x}{\eps_{j}})-\eps^{2}g^{kl}_{p}(x) \Gamma_{kl}^{s}(x)\partial_{s}U_{j}(\frac{x}{\eps_{j}}) &= V_{j}^{q}(\frac{x}{\eps_{j}}),\\
			\(-\eps^{2}\Delta +id\)W_{j}(\frac{x}{\eps_{j}}) &=V_{j}^{q}(\frac{x}{\eps_{j}}) \quad \text{in } B(0,\dfrac{R}{2}).
		\end{aligned}
	\end{equation*}
	This gives
	\begin{equation*}
		\(-\eps^{2}\Delta+id\) \Big[U_{j}(\frac{x}{\eps_{j}})- W_{j}(\frac{x}{\eps_{j}})\Big] =\eps^{2}g^{kl}_{p}(x) \Gamma_{kl}^{s}(x)\partial_{s}U_{j}(\frac{x}{\eps_{j}})-\eps^{2}(g^{kl}_{p}(x)-\delta_{kl})\partial_{kl}U_{j}(\frac{x}{\eps_{j}}) \quad \text{in } B(0,\dfrac{R}{2}).
	\end{equation*}
	From $L_{p}$-Estimate, we have
	\begin{equation*}
		\begin{aligned}
			&\left\Vert U_{j}\(\dfrac{x}{\eps_{j}}\)- W_{j}\(\dfrac{x}{\eps_{j}}\)\right\Vert_{W^{2,\tfrac{q+1}{q}}(B(0,\tfrac{R}{2}))}\\
			&\leq c\left\Vert \eps^{2}g^{kl}_{p}(x) \Gamma_{kl}^{s}(x)\partial_{s}U_{j}\(\dfrac{x}{\eps_{j}}\)-\eps^{2}(g^{kl}_{p}(x)-\delta_{kl})\partial_{kl}U_{j}\(\dfrac{x}{\eps_{j}}\)\right\Vert_{L^{\tfrac{q+1}{q}}(B(0,\tfrac{R}{2}))}\\
			&= c\left( \int_{B(0,\tfrac{R}{2})}\left|\eps_{j}^{2}g^{kl}_{p}(x) \Gamma_{kl}^{s}(x)\partial_{s}U_{j}\(\dfrac{x}{\eps_{j}}\)-\eps_{j}^{2}(g^{kl}_{p}(x)-\delta_{kl})\partial_{kl}U_{j}\(\dfrac{x}{\eps_{j}}\)\right| ^{\tfrac{q+1}{q}}dx\right) ^{\tfrac{q}{q+1}}\\
			&= c\eps_{j}^{\tfrac{Nq}{q+1}}\left( \int_{B(0,\tfrac{R}{2\eps_{j}})}\left|\eps_{j} g^{kl}_{p}(\eps_{j} x) \Gamma_{kl}^{s}(\eps_{j} x)\partial_{s}U_{j}(x)-(g^{kl}_{p}(\eps_{j} x)-\delta_{kl})\partial_{kl}U_{j}(x)\right| ^{\tfrac{q+1}{q}}dx\right) ^{\tfrac{q}{q+1}}\\
			&=O(\eps_{j}^{\tfrac{Nq}{q+1}+2}),
		\end{aligned}
	\end{equation*}
	where the estimate in last line follows from the expansions \eqref{metricexpansion} and \eqref{eqChristoffel} and the fact that $(U_j, V_j)$ are $C^2_{loc}$ with exponential decay \eqref{expoDecayR}. Now arguing similar to Step 2 of \cref{UE1a} we get
	\begin{equation}
		\begin{aligned}
			\int_{B_{g}(p_{j},\tfrac{R}{2})}u_{j}^{p}T_{1}v_{j}^{q}dv_g &=\int_{B(0,\tfrac{R}{2})}U_{j}^{p}\(\dfrac{x}{\eps_{j}}\)(-\eps^{2}\Delta+id)^{-1}V_{j}^{q}\(\dfrac{x}{\eps_{j}}\)\sqrt{g_{p_{j}}(x)}dx+o(\eps_{j}^{N+2}),\\
			\int_{B_{g}(p_{j},\tfrac{R}{2})}v_{j}^{q}T_{2}u_{j}^{p}dv_g
			&=\int_{B(0,\tfrac{R}{2})}V_{j}^{q}\(\dfrac{x}{\eps_{j}}\)(-\eps^{2}\Delta+id)^{-1}U_{j}^{p}\(\dfrac{x}{\eps_{j}}\)\sqrt{g_{p_{j}}(x)}dx+o(\eps_{j}^{N+2}).
		\end{aligned}
	\end{equation} 
	We also note that from an argument similar to \cref{UE1a}, we have
	\begin{equation*}
		\begin{aligned}
			\int_{B(0,\tfrac{R}{2\eps_{j}})} (U_{j}^{p}S_{1}V_{j}^{q}+V_{j}^{q}S_{2}U_{j}^{p})dy =\int_{B(0,\tfrac{R}{2\eps_{j}})} U_{j}^{p+1}+V_{j}^{q+1}dy+o(\eps_{j}).
		\end{aligned}
	\end{equation*}
	Combining all these terms we get
	\begin{equation*}
		\begin{aligned}
			&\int_{B_{g}(p_{j},\tfrac{R}{2})}(u_{j}^{p}T_{1}v_{j}^{q}+v_{j}^{q}T_{2}u_{j}^{p})dv_g\\
			&=\int_{B(0,\tfrac{R}{2})} \(U_{j}^{p}\(\dfrac{x}{\eps_{j}}\)(-\eps^{2}\Delta+id)^{-1}V_{j}^{q}\(\dfrac{x}{\eps_{j}}\)+V_{j}^{q}\(\dfrac{x}{\eps_{j}}\)(-\eps^{2}\Delta+id)^{-1}U_{j}^{p}\(\dfrac{x}{\eps_{j}}\)\)\sqrt{g_{p_{j}}(x)}dx+o(\eps_{j}^{N+2})\\
			&=\eps_{j}^{N}\int_{B(0,\tfrac{R}{2\eps_{j}})} \(U_{j}^{p}(y)(-\Delta+id)^{-1}V_{j}^{q}(y)+V_{j}^{q}(y)(-\Delta+id)^{-1}U_{j}^{p}(y)\)\sqrt{g_{p_{j}}(\eps y)}dy+o(\eps_{j}^{N+2})\\
			&=\eps_{j}^{N}\Big[\int_{B(0,\tfrac{R}{2\eps_{j}})} \(U_{j}^{p}S_{1}V_{j}^{q}+V_{j}^{q}S_{2}U_{j}^{p}\)dy-\int_{B(0,\tfrac{R}{2\eps_{j}})}\dfrac{1}{6}Ric(\eps_j y, \eps_j y)(U_{j}^{p+1}(y)+V_{j}^{q+1}(y))dy+o(\eps_{j}^{2})\Big]\\
		\end{aligned}
	\end{equation*}
	We now compare the estimate using Mountain Pass geometry to get the lower energy estimate.\\
	
	\noindent\textbf{Step II: Comparison using MP-Geometry:} Let us denote $(\tilde{U}_j, \tilde{V}_j):=(U_j^p, V_j^q)$. Then from Step (I), we have,
	\begin{align}
		\eps_{j}^{-N}H_j(t)&=\int_{\rN}\(\frac{t^\al}{\al}\tilde{U}_j^{\al}(y)+\frac{t^\beta}{\beta}\tilde{V}_j^{\beta}(y)\)dy-\frac{t^2}{2}\int_{\rN}\(\tilde{U}_j S_1\tilde{V}_j+\tilde{V}_jS_2\tilde{U}_j\)dy\notag\\
		&-\int_{\rN}\dfrac{1}{6}Ric(\eps_j y, \eps_j y)\Big[\(\frac{t^\al}{\al}-\frac{t^2}{2}\)\tilde{U}_j^{\al}(y)+\(\frac{t^\beta}{\beta}-\frac{t^2}{2}\)\tilde{V}_j^{\beta}(y)\Big]dy+o(\eps_j^2)\notag\\
		&=I_\infty(t(\tilde{U}_j,\tilde{V}_j))+G_j(t)+o(\eps_j^2)\label{loenes6}.
	\end{align}
	At the maximum point $t_j$ of $I_\infty(t(\tilde{U}_j,\tilde{V}_j))$, we have 
	\be \ml{C}_\infty\le I_\infty(t_j(\tilde{U}_j,\tilde{V}_j)).\label{loenes7}\ee
	Also we see that 
	\begin{equation*}
		\begin{aligned}
			\left.\frac{d}{dt} I_\infty(t(\tilde{U}_j,\tilde{V}_j))\right|_{t=1} &\to\int_{\rN}\(\tilde{U}^{\al}+\tilde{V}^{\beta}\)dy-\int_{\rN}\(\tilde{U} S_1\tilde{V}+\tilde{V}S_2\tilde{U}\)dy=0,\\
			\text{and }	\quad \left.\frac{d^2}{dt^2} I_\infty(t(\tilde{U}_j,\tilde{V}_j))\right| _{t=1}&\to \(\al-2\)\int_{\rN}|U|^{p+1}dx+\(\beta-2\)\int_{\rN}|V|^{q+1}dx+o(1)<0.
		\end{aligned}
	\end{equation*}
	Hence the unique maximum $t_j$ of $I_\infty(t(\tilde{U}_j,\tilde{V}_j))$ is of the form $t_j=1+o(1)$ as $\eps_j\to 0$. We now estimate the terms in $G_{j}(t)$.
	\begin{align*}
		&\int_{\rN}\dfrac{1}{6}Ric(\eps_j y, \eps_j y)\(\frac{t^\al}{\al}-\frac{t^2}{2}\)\tilde{U}_j^{\al}(y)dy\\
		&=\int_{\rN}\dfrac{1}{6}Ric(\eps_j y, \eps_j y)\(\frac{t^\al}{\al}-\frac{t^2}{2}\)U^p(y)dy+\int_{\rN}\dfrac{1}{6}Ric(\eps_j y, \eps_j y)\(\frac{t^\al}{\al}-\frac{t^2}{2}\)\(U_j^{p}(y)-U^{p}(y)\)dy.
	\end{align*}
	From the exponential decay of $U_j$ and $U$, together with the estimates $t_{j}=1+o(1)$, $Ric(\eps_{j} y, \eps_{j} y)=O(\eps_{j}^{2}|y|^2)$, and the convergence $U_{j}\to U$ in $C^2_{loc}(\rN)$, we have that the second term in the above integral is of $o(\eps^{2}_{j})$. Since $U$ is radial, we have (see Proposition 3.4 of \cite{MR2180862})
	\begin{align*}
	 	\int_{\rN}\dfrac{1}{6}Ric(\eps_j y, \eps_j y)\(\frac{t^\al}{\al}-\frac{t^2}{2}\)U^p(y)dy=\eps_{j}^{2}\dfrac{S(p_{j})}{6N}\int_{\rN}\frac{p-1}{2(p+1)}U^{p+1}(y)dy+o(\eps^2_j).
	 \end{align*}
	 Calculating similarly for the integral corresponding to $\tilde{V}_j$, we conclude
	\begin{align}
		G_j(t_j)&=-\eps_{j}^{2}\dfrac{S(p_{j})}{6N}\int_{\rN}\Big[\frac{p-1}{2(p+1)}U^{p+1}(y)+\frac{q-1}{2(q+1)}V^{q+1}(y)\Big]|y|^{2}dy+o(\eps_j^2).\label{loenes8}
	\end{align}
	Combining \ef{loenes2}, \ef{loenes6}, \ef{loenes7}, and \ef{loenes8} we conclude
	\begin{equation*}
		c_{\eps_{j}}\geq\eps_{j}^{N}\left[ \ml{C}_\infty-\eps_{j}^{2}\dfrac{S(p_{j})}{6N}\left(\dfrac{p-1}{2(p+1)}\int_{\R^{N}}U^{p+1}(y)|y|^{2}dy+\dfrac{q-1}{2(q+1)}\int_{\R^{N}}V^{q+1}(y)|y|^{2}dy\right)+o(\eps_{j}^{2})\right].
	\end{equation*}
\end{proof}	
We now prove our main result Theorem \ref{Thm}:
\begin{proof}
	Let us denote
	\begin{equation*}
		\eta(U,V):=\dfrac{p-1}{2(p+1)}\int_{\rN}U^{p+1}(y)|y|^{2}dy-\dfrac{q-1}{2(q+1)}\int_{\rN}V^{q+1}\(y\)|y|^{2}dy.
	\end{equation*}
	From \cref{cptsolnsp} we have that both $\max_{(W,Z)\in \ml{L}}\eta(W,Z)$ and $\min_{(W,Z)\in \ml{L}}\eta(W,Z)$ are finite. Now from the Upper energy estimate \cref{UEEa}, we have the following:
	
	If $max_{x\in\M}S(x)\ge 0$, we have 
	\be c_{\eps} \leq \eps^{N}\left[ \ml{C}_\infty-\dfrac{\eps^{2}}{6N}\max_{p\in \M}S(p)\max_{(W,Z)\in \ml{L}}\eta(W,Z)+ o(\eps^{2})\right].\label{upcon1}\ee
	And if $max_{x\in\M}S(x)\le 0$, we have
	\be c_{\eps} \leq \eps^{N}\left[ \ml{C}_\infty-\dfrac{\eps^{2}}{6N}\max_{p\in \M}S(p)\min_{(W,Z)\in \ml{L}}\eta(W,Z)+ o(\eps^{2})\right].\label{upcon2}\ee
	Using the Lower energy estimate \cref{LEEa} we get, for  $max_{x\in\M}S(x)\ge 0$,
	\be S(p_j)\eta(U,V)\ge \max_{p\in \M}S(p)\max_{(W,Z)\in \ml{L}}\eta(W,Z)+o(1).\nee
	And if $max_{x\in\M}S(x)\le 0$,
	\be S(p_j)\eta(U,V)\ge \max_{p\in \M}S(p)\min_{(W,Z)\in \ml{L}}\eta(W,Z)+o(1).\nee
	Therefore we conclude that 
	\be S(p_\eps) \to \max_{p\in \M}S(p).\nee
\end{proof}

\appendix

\section{Boot-strap Argument for the Hamiltonian system}\label{APPB}

Let $(u_\eps,v_\eps)\in W^{2,\beta}(\M)\times W^{2,\al}(\M)$ be a solution of \ef{eq0}. We prove that $u_\eps,v_\eps\in C^{2,\al}(M)$ for some $0<\al<1$. For convenience, we omit the subscript and denote the solution by $(u, v)$. A similar regularity result can be found in the proof of Theorem 1 in \cite{MR1785681}.

\textbf{Step 0:} We shall construct an increasing sequence.
First note by Sobolev embedding, we have $(u, v) \in L^{p_{1}}(\M) \times L^{q_{1}}(\M)$, where 
\be p_{1}=\beta^*=\dfrac{N(q+1)}{Nq-2(q+1)} \text{ and } q_{1}=\al^*=\dfrac{N(p+1)}{Np-2(p+1)}.\nee

It follows from \eqref{HC} that 
\begin{equation*}
	\frac{1}{p+1}+\frac{1}{q+1}>\frac{N-2}{N} \iff p+1<p_{1} \text{ and } q+1<q_{1}.
\end{equation*}

If $Nq-2(q+1)\leq 0$ and $Np-2(p+1) \leq 0$, we have nothing to prove. Assume not.

\textbf{Step 1: } Next step we have $|v|^{q-1}v\in L^{\frac{q_1}{q}}(\M)$ which again implies (using \cref{LP-estimate}) that $u\in W^{2,\frac{q_1}{q}}(\M)$. Using Sobolev Embedding, we have $u \in L^{p_{2}}(\M)$, where \be p_{2}=\frac{Nq_1}{Nq-2q_1}.\nee This gives $|u|^{p-1}u\in L^{\frac{p_2}{p}}(\M)$ which again implies (using \cref{LP-estimate}) that $v\in W^{2,\frac{p_2}{p}}(\M)$. Again, using Sobolev embedding, we have $v\in L^{q_{2}}(\M)$, where 
\be q_{2}=\frac{Np_2}{Np-2p_2}.\nee

Now we see that $q_1>q+1$ implies
\begin{align*}
	p_2-p_1=\frac{Nq_1}{Nq-2q_1}-\dfrac{N(q+1)}{Nq-2(q+1)}>0.
\end{align*}
Also,
\begin{align*}
	q_2-q_1=\frac{Np_2}{Np-2p_2}-q_1=\frac{N\frac{Nq_1}{Nq-2q_1}}{Np-2\frac{Nq_1}{Nq-2q_1}}-q_1=q_1\(\frac{N}{Npq-2q_1(p+1)}-1\).
\end{align*}
Hence, 
\begin{align*}
	q_2-q_1>0 \Longleftrightarrow N>Npq-2q_1(p+1) \Longleftrightarrow q_1>\frac{N(pq-1)}{2(p+1)}.
\end{align*}
Let us take $h_{q}(p)=\frac{N(pq-1)}{2p+2}$. Now from the condition \ef{HC} we have
\be 1<p< \frac{N+2(q+1)}{Nq-2(q+1)}=\bar{p} \text{ say,}\nee 
Then,
\begin{equation*}
	\begin{aligned}
		h_{q}^{\prime}(p)
		&=\dfrac{N(q+1)}{2(p+1)^{2}}>0,\\
	\end{aligned}
\end{equation*}
implies the maximum value of $h_q(p)$ is $h_q(\bar{p})=q+1.$ This gives $q_1>q+1>h_q(p)=\frac{N(pq-1)}{2p+2}$ and $q_{2}>q_{1}$.

\textbf{Step 2: } Now we have $|u|^{p-1}u\in L^{\frac{p_2}{p}}(\M)$ and $|v|^{q-1}v\in L^{\frac{q_2}{q}}(\M)$ which again implies that $u\in W^{2,\frac{q_2}{q}}(\M)$ and $v\in W^{2,\frac{p_2}{p}}(\M)$. Again, using Sobolev Embedding, we have $(u, v) \in L^{p_{3}}(\M) \times L^{q_{3}}(\M)$, where 
\be p_{3}=\frac{Nq_2}{Nq-2q_2} \text{ and } q_{3}=\frac{Np_2}{Np-2p_2}.\nee

Continuing this process we construct two sequences $p_{n}$ and $q_{n}$ such that
\begin{equation*}
	\begin{aligned}
		p_{n+1}&=\dfrac{Nq_{n}}{Nq-2q_{n}} \qquad \text{ and }\qquad  q_{n+1}&=\dfrac{Np_{n+1}}{Np-2p_{n+1}}.
	\end{aligned}
\end{equation*}

It is easy to see that, if $p_n>p_{n-1}$ and $q_n>q_{n-1}$ then from the above expression $p_{n+1}>p_{n}$ and $q_{n+1}>q_{n}.$ 

\textbf{Step 3: } We shall show that $p_n,q_n\to \infty.$ If possible let $p_n\to p_0$ then we see that $q_n \to \dfrac{Np_0}{Np-2p_0}=q_0$ say. And from the above expression we get  
\begin{equation*}
	p_0=\dfrac{Nq_0}{Nq-2q_0}, \qquad q_0=\dfrac{Np_0}{Np-2p_0}.
\end{equation*} 
Thus, 
\begin{equation*}
	\begin{aligned}
		q_0&=\dfrac{\frac{N^{2}q_0}{Nq-2q_0}}{Np-\frac{2Nq_0}{Nq-2q_0}}
		=\dfrac{Nq_0}{p(Nq-2q_0)-2q_0}
		\implies Npq-2pq_0-2q_0=N\\
		&\implies N(pq-1)=q_0(2p+2)\implies q_0= \frac{N(pq-1)}{2p+2}<q+1\Rightarrow\Leftarrow.
	\end{aligned}
\end{equation*}
Hence both $p_n$ and $q_n$ converges to infinity. Then by Sobolev embedding and Schauder estimate we have both $u_\eps,v_\eps\in C^{2,\al}(M)$ for some $0<\al<1$.

\section*{Acknowledgment}
The first author, Anusree R is supported by the Ministry of Education, Govt. of India. The second author is supported by SERB-MATRICS grant, the Ministry of Education, Govt. of India.

\section*{Data availability statement:} Not applicable as no new datasets were created on analyzed in this study.


\bibliographystyle{plain}

\end{document}